\documentclass{amsart}%
\usepackage{amsfonts}
\usepackage{amsmath}
\usepackage{amssymb}
\usepackage{graphicx}%
\providecommand{\U}[1]{\protect\rule{.1in}{.1in}}
\providecommand{\U}[1]{\protect\rule{.1in}{.1in}}
\theoremstyle{definition}
\newtheorem{definition}{Definition}[section]

\newtheorem{remark}[definition]{Remark}

\newtheorem*{Exm}{Examples}
\theoremstyle{plain}

\newtheorem{theorem}[definition]{Theorem}
\newtheorem{proposition}[definition]{Proposition}
\newtheorem{lemma}[definition]{Lemma}
\newtheorem{corollary}[definition]{Corollary}
\numberwithin{equation}{section}
\begin{document}
\author[S.\,Biagi]{Stefano Biagi}
\author[M.\,Bramanti]{Marco Bramanti}
\address{\textnormal{S.\,Biagi and M.\,Bramanti}: Politecnico di Milano - Dipartimento di
Matematica \\
\indent
Via Bonardi 9, 20133 Milano, Italy}
\email[S.\,Biagi]{stefano.biagi@polimi.it}
\email[M.\,Bramanti]{marco.bramanti@polimi.it}

\begin{abstract}
Let $X_{1},...,X_{m}$ be a family of real smooth vector fields defined in
$\mathbb{R}^{n}$, $1$-homogeneous with respect to a nonisotropic family of
dilations and satisfying H\"{o}rmander's rank condition at $0$ (and therefore
at every point of $\mathbb{R}^{n}$). The vector fields are \emph{not} assumed
to be translation invariant with respect to any Lie group structure. Let us
consider the non\-va\-ria\-tio\-nal evolution operator
\[
\mathcal{H}:=\sum_{i,j=1}^{m}a_{i,j}(t,x)X_{i}X_{j}-\partial_{t}
\]
where $(a_{i,j}(t,x))_{i,j=1}^{m}$ is a symmetric uniformly positive $m\times
m$ matrix and the entries $a_{ij}$ are bounded H\"{o}lder continuous functions
on $\mathbb{R}^{1+n}$, with respect to the ``parabolic'' distance induced by
the vector fields. We prove the existence of a global heat kernel
$\Gamma(\cdot;s,y)\in C_{X,\mathrm{loc}}^{2,\alpha}(\mathbb{R}^{1+n}%
\setminus\{(s,y)\})$ for $\mathcal{H}$, such that $\Gamma$ satisfies two-sided
Gaussian bounds and $\partial_{t}\Gamma, X_{i}\Gamma,X_{i}X_{j}\Gamma$ satisfy
upper Gaussian bounds on every strip $[0,T]\times\mathbb{R}^{n}$. We also
prove a scale-invariant parabolic Harnack i\-ne\-qua\-li\-ty for $\mathcal{H}%
$, and a standard Harnack inequality for the corresponding stationary
operator
\[
\mathcal{L}:=\sum_{i,j=1}^{m}a_{i,j}(x)X_{i}X_{j}.
\]
with H\"{o}lder continuous coefficients.

\end{abstract}
\title[Operators structured on homogeneous H\"{o}rmander vector fields]{Non-divergence operators structured on homogeneous H\"{o}rmander vector
fields: heat kernels and global Gaussian bounds}
\date{\today }
\maketitle

\section{Introduction}

Let $X_{1},...,X_{m}$ be a family of real smooth vector fields defined in some
domain $\Omega\subseteq\mathbb{R}^{n}$, satisfying H\"{o}rmander's rank
condition in $\Omega$. We consider the non\-va\-ria\-tio\-nal evolution
operator
\begin{equation}
\mathcal{H}:=\sum_{i,j=1}^{m}a_{i,j}(t,x)X_{i}X_{j}-\partial_{t} \label{H var}%
\end{equation}
where $(a_{i,j}(t,x))_{i,j=1}^{m}$ is a symmetric uniformly positive $m\times
m$ matrix and the entries $a_{ij}$ are bounded H\"{o}lder continuous functions
on $[0,T]\times\Omega$. (Pre\-ci\-se de\-fi\-ni\-tions will be given later).

Motivated by issues arising in the theory of several complex variables,
operators of this kind have been studied by several Authors. Bonfiglioli,
Lanconelli, Uguzzoni, in a series of papers (\cite{BLUpaper}, \cite{BLUpap2},
\cite{BU}, \cite{BU2}, \cite{BoUg}) have carried out the following research
program. Given a set of $1$-homogeneous, left-invariant H\"{o}rmander vector
fields on a Carnot group $\mathbb{G} = (\mathbb{R}^{n},*)$, they have proved
in \cite{BLUpap2} the existence of a global heat kernel $\Gamma$ for
$\mathcal{H}$, satisfying sharp Gaussian bounds, of the form
\begin{equation}
\label{eq.estimBLU}%
\begin{split}
&  \mathbf{c}^{-1}(t-s)^{-Q/2}\exp\Big(-M\frac{\|\xi^{-1}*x\|^{2}}%
{t-s}\Big) \leq\Gamma(t,x;s,y)\\
&  \qquad\qquad\leq\mathbf{c}(t-s)^{-Q/2}\exp\Big(-\frac{\|\xi^{-1}*x\|^{2}%
}{M(t-s)}\Big),
\end{split}
\end{equation}
and analogous upper estimates for the first and second order derivatives of
$\Gamma$ along $X_{1},\ldots,X_{m}$. In \eqref{eq.estimBLU}, $Q$ and
$\|\cdot\|$ are, respectively, the homogeneous dimension and a homogeneous
norm on $\mathbb{G}$. Exploiting these results, they have derived in
\cite{BoUg} a scale invariant parabolic Harnack inequality for $\mathcal{H}$,
which easily implies an analogous standard Harnack inequality for the
corresponding stationary operator
\[
\mathcal{L}:=\sum_{i,j=1}^{m}a_{i,j}(x)X_{i}X_{j}.
\]
In order to build the heat kernel $\Gamma$ for $\mathcal{H}$, the Authors
exploit the \emph{parametrix method}. This requires much preliminary work on
the corresponding constant coefficient operator
\[
\mathcal{H}_{A}:=\sum_{i,j=1}^{m}a_{i,j}X_{i}X_{j}-\partial_{t}
\]
where $A=\left\{  a_{ij}\right\}  $ belongs to the class $\mathcal{M}%
_{\Lambda}$ of constant symmetric matrices satisfying
\begin{equation}
\label{ellipticity}\frac{1}{\Lambda}|\xi|^{2}\leq\langle A\xi,\xi\rangle
\leq\Lambda|\xi|^{2}\text{ \ for every }\xi\in\mathbb{R}^{m}\text{.}%
\end{equation}
Namely, in \cite{BLUpaper} the Authors have proved sharp Gaussian estimates
for the heat kernel $\Gamma_{A}$ of $\mathcal{H}_{A}$, where the bounds depend
on $A$ only through the number $\Lambda$. In turn, the desired uniformity of
the estimates relies on a careful analysis of a diffeomorphism turning the
operator $\mathcal{H}_{A}$ into $\mathcal{H}_{I}$ (with $I$ the identity
matrix), carried out in \cite{BU} and also exploiting the results of
\cite{BU2}.

Bramanti, Brandolini, Lanconelli, Uguzzoni, in \cite{BBLUMemoir}, have studied
heat-type operators $\mathcal{H}$ with\-out assuming the existence of an
underlying Carnot group. In other words, the vector fields $X_{1},\ldots
,X_{m}$ are now a general family of H\"{o}rmander vector fields. On the other
hand, the operator $\mathcal{H}$ is assumed to coincide with the classical
heat operator outside a large compact set. Under these assumptions, the
Authors have implemented the same general research program described above:
after establishing uniform Gaussian estimates for operators $\mathcal{H}_{A}$
corresponding to a contant matrix $A$, by the parametrix method a global
Gaussian kernel is built for $\mathcal{H}$, and sharp Gaussian bounds are
established, of the kind
\[%
\begin{split}
&  \mathbf{c}^{-1}\frac{1}{|B_{X}(x,\sqrt{t-s})|} \exp\Big(-M\frac{d_{X}%
(x,y)}{t-s}\Big) \leq\Gamma(t,x;s,y)\\
&  \qquad\qquad\leq\mathbf{c}\frac{1}{|B_{X}(x,\sqrt{t-s})|}\exp
\Big(-\frac{d_{X}(x,y)}{M(t-s)}\Big),
\end{split}
\]
with analogous upper estimates for the $X$-derivatives of $\Gamma$. Here,
$d_{X}$ is the control distance induced by $X_{1},\ldots,X_{m}$ and
$B_{X}(x,r)$ is the corresponding ball. As a con\-se\-quence, scale invariant
Harnack inequalities for $\mathcal{H}$ and $\mathcal{L}$ are derived. The
results in \cite{BBLUMemoir} exploit, in particular, both some of the
corresponding results proved on Carnot groups in the aforementioned papers by
Bonfiglioli, Lanconelli, Uguzzoni, and Schauder-type estimates for
$\mathcal{H}$ proved by Bramanti, Brandolini in \cite{BraBra}.

Since the vector fields considered in \cite{BBLUMemoir} are not assumed to be
homogeneous nor left invariant with respect to an underlying group structure,
under this respect that theory is more general than the one developed by
Bonfiglioli, Lanconelli, Uguzzoni. On the other hand, the requirement that
$\mathcal{H}$ coincides with the heat operator outside a compact set means
that the results proved in \cite{BBLUMemoir} are actually \emph{local
results}, although they are better formulated with the language of a globally
defined operator. This fact is consistent with a quite pervasive dichotomy in
the theory of H\"{o}rmander operators: global results in the setting of Carnot
groups \emph{versus} local results in the general setting.

The aim of this paper is to establish the same set of results (i.e.: existence
of a global heat kernel $\Gamma$ for $\mathcal{H}$, sharp Gaussian bounds on
$\Gamma$, scale invariant Harnack inequalities for $\mathcal{H}$ and
$\mathcal{L}$), in a \emph{global }version, for a family of H\"{o}rmander
vector fields more general than the generators of a Carnot group. A convenient
setting is that of smooth vector fields $X_{1},\ldots,X_{m}$ in $\mathbb{R}%
^{n}$ satisfying the next assumptions:

\begin{description}
\item[(H.1)] $X_{1},\ldots,X_{m}$ are linearly independent (as vector fields)
and ho\-mo\-ge\-neous of degree $1$ with respect to a family of non-isotropic
dilations $\{\delta_{\lambda}\}_{\lambda>0}$ in $\mathbb{R}^{n}$ of the
following form
\begin{equation}
\label{eq.defidela}\delta_{\lambda}(x):=(\lambda^{\sigma_{1}}x_{1}%
,\ldots,\lambda^{\sigma_{n}}x_{n}),\quad%
\begin{array}
[c]{c}%
\text{where $\sigma_{1},\ldots,\sigma_{n}\in\mathbb{N}$ and}\\
1=\sigma_{1}\leq\ldots\leq\sigma_{n}%
\end{array}
\end{equation}
We define the $\delta_{\lambda}$-homogeneous dimension of $(\mathbb{R}%
^{n},\delta_{\lambda})$ as
\begin{equation}
\textstyle q:=\sum_{k=1}^{n}\sigma_{k}\geq n. \label{eq.defqhom}%
\end{equation}

\item[(H.2)] $X_{1},\ldots,X_{m}$ satisfy H\"{o}rmander's condition at $x=0$,
that is,
\[
\mathrm{dim}\big\{  Y(0):\,Y\in\mathrm{Lie}(X_{1},\ldots,X_{m})\big\}  =n,
\]
where $\mathrm{Lie}(X_{1},\ldots,X_{m})$ is the \emph{Lie algebra generated by
$X_{1},\ldots,X_{m}$}.
\end{description}

Some  examples of vector fields of this kind are the following.

\begin{Exm}
(1)\,\,In $\mathbb{R}^{2}$:
\[
\text{$X_{1} = \partial_{x_{1}}$ and $X_{2} = x_{1}^{k}\,\partial_{x_{2}}$}%
\]
(with $k\in\mathbb{N}$), which are $1$-homogeneous with respect to
$\delta_{\lambda}(x) = (\lambda x_{1},\lambda^{k+1}x_{2}).$ \medskip

\noindent(2)\,\,In $\mathbb{R}^{n}$:
\[
\text{$X_{1} = \partial_{x_{1}}$ and $X_{2} = x_{1}\partial_{x_{2}} +
x_{2}\partial_{x_{3}}+\ldots+x_{n-1}\partial_{x_{n}}$},
\]
with $\delta_{\lambda}(x) =(\lambda x_{1},\lambda^{2}x_{2},\cdots,\lambda
^{n}x_{n}).$ \medskip

\noindent(3)\,\,In $\mathbb{R}^{3}$:
\[
\text{ $X_{1} = \partial_{x_{1}}$ and $X_{2} = x_{1}\,\partial_{x_{2}}%
+x_{1}^{2}\,\partial_{x_{3}}$},
\]
with $\delta_{\lambda}(x) = (\lambda x_{1},\lambda^{2} x_{2},\lambda^{3}
x_{3})$. \medskip

\noindent(4)\,\,In $\mathbb{R}^{n}$:
\[
\text{$X_{1} = \partial_{x_{1}}$ and $X_{2} = x_{1}\,\partial_{x_{2}}%
+x_{1}^{2}\,\partial_{x_{3}}+ \cdots+x_{1}^{n-1}\,\partial_{x_{n}}$},
\]
with the same dilations as in (2).
\end{Exm}

Under these as\-sump\-tions, H\"{o}rmander operators of the kind
\[
L=\sum_{i=1}^{m}X_{i}^{2}
\]
or their evolutive counterpart
\[
H =\partial_{t}-\sum_{i=1}^{m}X_{i}^{2}
\]
have been studied in recent years in a series of papers. Biagi, Bonfiglioli in
\cite{BiagiBonfiglioliLast} and \cite{BBHeat} have proved the existence of a
homogeneous global fundamental solution for $L$ and $H$, respectively. Their
technique consists in constructing a higher dimensional Carnot group whose
generators $\widehat{X}_{1},...\widehat{X}_{m}$ project onto $X_{1},...X_{m}$.
The cor\-re\-spon\-ding \textquotedblleft lifted\textquotedblright\ operators
$\widehat{L},\widehat{H}$, by known results possess a homogeneous fundamental
solution; integrating these kernels with respect to the added variables, the
Authors get, and are able to estimate, global homogeneous fundamental
solutions for $L$ and $H$. More explicit bounds on these fundamental
solutions, in terms of the distance induced by the vector fields and the
volume of the corresponding balls, have been proved in \cite{BiBoBra1} and
\cite{BiBraGauss}, respectively.

By combining this lifting technique with the bounds proved in
\cite{BiBraGauss}, we first establish \emph{sharp uniform} Gaussian bounds for
the heat kernels cor\-re\-spon\-ding to operators \eqref{H var} with constant
$a_{ij}$ (see Theorem \ref{thm.mainParabolicCase}); next, we show that the
pa\-ra\-me\-trix method is applicable as in \cite{BBLUMemoir}, getting the
existence and sharp Gaussian bounds for the heat kernel of operators
\eqref{H var} with \emph{H\"{o}lder continuous coefficients}. Before giving
the precise statement of this result, which is one of the main results of this
paper, it is convenient to fix the following: \vspace*{0.1cm}

\begin{definition}
Let $\Omega\subseteq\mathbb{R}^{1+n}=\mathbb{R}_{t}\times\mathbb{R}_{x}^{n}$
be an open set, and $\alpha\in(0,1)$. We define ${C}_{X}^{\alpha}(\Omega)$ as
the space of functions $u:\Omega\rightarrow\mathbb{R}$ such that
\[
\Vert u\Vert_{\alpha,\Omega}:=\sup_{\Omega}|u|+\sup_{\begin{subarray}{c} 
 (t,x),(s,y)\in
\Omega \\
 (t,x)\neq(s,y)
 \end{subarray}}\frac{|u(t,x)-u(s,y)|}{d_{X}(x,y)^{\alpha
}+|t-s|^{\alpha/2}}<\infty
\]
where $d_{X}$ is the Carnot-Carath\'{e}odory distance associated with
\[
X:=\{X_{1},\ldots,X_{m}\}
\]
(see Definition \ref{def.CCdistance} in Section \ref{sec.prelim}).
Accordingly, we define ${C}_{X}^{2,\alpha}(\Omega)$ as the space of
fun\-ctions $u:\Omega\rightarrow\mathbb{R}$ such that
\[
\text{$u,\,X_{i}u,\,X_{i}X_{j}u$ and $\partial_{t}u\in{C}_{X}^{\alpha}%
(\Omega)$},
\]
where all the $X$-derivatives exist in the \emph{intrinsic sense}. Finally, we
define ${C}_{X,\mathrm{loc}}^{2,\alpha}(\Omega)$ as the space of functions
$u:\Omega\rightarrow\mathbb{R}$ such that
\[
\text{$u|_{V}\in{C}_{X}^{2,\alpha}(V)$ for every open set $V\Subset\Omega$}.
\]

\end{definition}

With the above definition at hand, we can now state the announced result
providing \emph{existence and properties} of the global heat kernel of
$\mathcal{H}$.

\begin{theorem}
[Heat kernel for $\mathcal{H}$]\label{thm.thmVariableCoeff} Let $X_{1}%
,\ldots,X_{m}\ $be a family of linearly in\-de\-pen\-dent, smooth vector
fields in $\mathbb{R}^{n}$, ho\-mo\-ge\-neous of degree $1$ with respect to a
family of non-isotropic dilations $\{\delta_{\lambda}\}_{\lambda>0}$ of the
form \eqref{eq.defidela}. Assume that $X_{1},\ldots,X_{m}$ satisfy
H\"{o}rmander's rank condition at $0$ \emph{(}and therefore at every point of
$\mathbb{R}^{n}$, as will be explained later\emph{)}. Moreover, let
\[
A(t,x)=(a_{i,j}(t,x))_{i,j=1}^{m}%
\]
be a symmetric $m\times m$ matrix of functions such that:

\begin{itemize}
\item[(i)] $a_{i,j}\in{C}_{X}^{\alpha}(\mathbb{R}^{1+n})$ for every
$i,j=1,\ldots,m$;

\item[(ii)] the following \emph{uniform ellipticity condition holds}: there
exists $\Lambda>1$ s.t.
\[
\frac{1}{\Lambda}|\xi|^{2}\leq\langle A(t,x)\xi,\xi\rangle\leq\Lambda|\xi
|^{2}\qquad\text{for every $\xi\in\mathbb{R}^{m}$},(t,x)\in\mathbb{R}^{1+n}.
\]

\end{itemize}

Let $\mathcal{H}$ be as in \eqref{H var}. Then, there exists a function
\emph{(}\textquotedblleft heat kernel\textquotedblright\ for $\mathcal{H}%
$\emph{)}
\[
\Gamma:\mathbb{R}^{1+n}\times\mathbb{R}^{1+n}\rightarrow\mathbb{R}%
,\qquad\Gamma=\Gamma(t,x;s,y),
\]
which satisfies the properties listed below.

\begin{enumerate}
\item $\Gamma$ is continuous out of the diagonal of $\mathbb{R}^{1+n}%
\times\mathbb{R}^{1+n}$.

\item $\Gamma(t,x;s,y)$ is non-negative, and it vanishes for $t\leq s$.

\item For every fixed $(s,y)\in\mathbb{R}^{1+n}$ we have
\begin{align*}
&  \Gamma(\cdot;s,y)\in{C}_{X,\mathrm{loc}}^{2,\alpha}(\mathbb{R}%
^{1+n}\setminus\{(s,y)\})\quad\text{and}\\
&  \text{$\mathcal{H}(\Gamma(\cdot;s,y))=0$ on $\mathbb{R}^{1+n}%
\setminus\{(s,y)\}$};
\end{align*}

\item For every $T>0$ there exists a constant $\mathbf{c}=\mathbf{c}_{T}>0$
such that
\begin{align*}
&  \mathrm{(i)}\,\,\mathbf{c}^{-1}\frac{1}{|B_{X}(x,\sqrt{t-s})|}\,\exp\left(
-\mathbf{c}\frac{d_{X}(x,y)^{2}}{t-s}\right)  \leq\Gamma(t,x;s,y)\\
&  \qquad\qquad\leq\mathbf{c}\frac{1}{|B_{X}(x,\sqrt{t-s})|}\,\exp\left(
-\frac{d_{X}(x,y)^{2}}{\mathbf{c}(t-s)}\right)  ;\\[0.2cm]
&  \mathrm{(ii)}\,\,|X_{i}(\Gamma(\cdot;s,y))(t,x)|\leq\mathbf{c}\frac
{1}{\sqrt{t-s}|B_{X}(x,\sqrt{t-s})|}\,\exp\left(  -\frac{d_{X}(x,y)^{2}%
}{\mathbf{c}(t-s)}\right)  ;\\[0.2cm]
&  \mathrm{(iii)}\,\,|X_{i}X_{j}(\Gamma(\cdot;s,y))(t,x)|+|\partial_{t}%
(\Gamma(\cdot;s,y))(t,x)|\\
&  \qquad\qquad\leq\mathbf{c}\frac{1}{\left(  t-s\right)  |B_{X}(x,\sqrt
{t-s})|}\,\exp\left(  -\frac{d_{X}(x,y)^{2}}{\mathbf{c}(t-s)}\right)  ;
\end{align*}
for every $(t,x),(s,y)\in\mathbb{R}^{1+n}$ with $0<t-s\leq T$ and $1\leq
i,j\leq m$.

\item There exists a constant $\delta>0$ such that the following assertion
holds. \vspace*{0.05cm}

Let $\mu\geq0$ and let $T>0$ satisfy
\[
T\mu<\delta.
\]
Moreover, let $f\in{C}_{X}^{\alpha}([0,T]\times\mathbb{R}^{n})$ and let $g\in
C(\mathbb{R}^{n})$ be such that
\[
|f(t,x)|,\,|g(x)|\leq M\exp\left(  \mu\,d_{X}(x,0)^{2}\right)  ,
\]
for all $(t,x)\in\lbrack0,T]\times\mathbb{R}^{n}$ and for some $M>0$. Then,
the function
\begin{align*}
u(t,x):  &  =\int_{\mathbb{R}^{n}}\Gamma(t,x;0,y)g(y)\,dy\\
&  \qquad+\int_{[0,t]\times\mathbb{R}^{n}}\Gamma(t,x;s,y)f(s,y)\,dsdy
\end{align*}
is well-defined on $[0,T]\times\mathbb{R}^{n}$, and enjoys the following
properties: \medskip

\emph{(i)}\thinspace\thinspace$u\in{C}_{X,\mathrm{loc}}^{2,\alpha}%
((0,T)\times\mathbb{R}^{n})\cap C([0,T]\times\mathbb{R}^{n})$; \vspace*{0.1cm}

\emph{(ii)}\thinspace\thinspace$u$ solves the Cauchy problem
\[%
\begin{cases}
\mathcal{H}u=f & \text{in $(0,T)\times\mathbb{R}^{n}$},\\
u(0,\cdot)=g & \text{in $\mathbb{R}^{n}$}.
\end{cases}
\]

\item The following reproduction formula holds
\[
\Gamma(t,x;s,y)=\int_{\mathbb{R}^{n}}\Gamma(t,x;\tau,\xi)\Gamma(\tau
,\xi;s,y)\,d\xi,
\]
for every $x,y\in\mathbb{R}^{n}$ and $t>\tau>s$.

\item Suppose, in addition, that the functions $a_{i,j}$ are smooth on
$\mathbb{R}^{1+n}$. Then, the operator $\mathcal{H}$ is $C^{\infty}%
$-hypoelliptic in $\mathbb{R}^{1+n}$, and
\[
\mathcal{H}(\Gamma(\cdot;s,y))=-\mathrm{\delta}_{(s,y)}\quad\text{in
$\mathcal{D}^{\prime}(\mathbb{R}^{1+n})$},
\]
where $\mathrm{\delta}_{(s,y)}$ denotes the Dirac delta centered at $(s,y)$.
\end{enumerate}
\end{theorem}

As anticipated, the second main result of this paper is a
\emph{scale-invariant Har\-nack inequality} for $\mathcal{H}$,  and it will be
stated in Section \ref{sec.HarnackH}, see Theorem \ref{thm.Harnack}.  The
proof of this result (and that of its stationary counterpart,  see Theorem
\ref{thm.HarnackStat}) can follow an easier path, logically  independent of
the properties of the heat kernel: the Harnack inequality can  be simply
derived from the corresponding result which is known in Carnot  groups, just
by projection, owing to the lifting procedure sketched above. We  note that
this projection technique would not, instead, allow to get a simple  proof of
the existence of a global fundamental solution for \eqref{H var}.

\section{Assumptions, notation and preliminary results}

\label{sec.prelim} Here we explain and discuss in detail the notions and
assumptions involved in the statement of Theorem \ref{thm.thmVariableCoeff}.
To begin with, we point out some easy consequences of as\-sum\-ptions
(H.1)-(H.2) which will be useful in the sequel (for a proof see
\cite{BiagiBonfiglioliLast}). 

\begin{enumerate}

\item H\"{o}rmander's condition holds at every  point $x\in\mathbb{R}^{n}$,
i.e.,
\[
\mathrm{dim}\big\{  Y(x):\,Y\in\mathrm{Lie}(X_{1},\ldots,X_{m}%
)\big\}  =n\qquad\text{ for every $x\in\mathbb{R}^{n}$}.
\]

\item The Lie algebra $\mathrm{Lie}(X_{1},\ldots,X_{m})$ is \emph{nilpotent}
and  \emph{stratified}, that is
\begin{equation}
\label{eq.LieXoplus}\mathrm{Lie}(X_{1},\ldots,X_{m})=\bigoplus_{i=1}%
^{\sigma_{n}}V_{i},
\end{equation}
where $V_{1}=\mathrm{span}\{X_{1},\ldots,X_{m}\}$ and $V_{i}:=[V_{1}%
,V_{i-1}]$  (for $i\geq2$). Fur\-ther\-mo\-re, for every $i=1,\ldots
,\sigma_{n}$ one also  has
\begin{equation}
\label{eq.Viintrinsic}V_{i}=\big\{ Y\in\mathrm{Lie}(X_{1},\ldots
,X_{m}):\,\text{$Y$ is $\delta_{\lambda}$-homogeneous of degree $i$}\big\} .
\end{equation}
As a consequence, since it is finitely-generated, $\mathrm{Lie}(X_{1}%
,\ldots,X_{m})$ has  \emph{finite dimension}, say $N$. Moreover, using
assumption (H.2), one gets
\begin{equation}
\label{eq.dimNgeqn}N=\mathrm{Lie}(X_{1},\ldots,X_{m})\geq n.
\end{equation}

\end{enumerate}

From now on, we will adopt the simplified notation
\[
\mathfrak{a}:=\mathrm{Lie}(X_{1},\ldots,X_{m}),
\]
so that $N:=\mathrm{dim}(\mathfrak{a})$. On account of \eqref{eq.dimNgeqn},
only two cases can occur:

\begin{itemize}
\item[(a)] $N=n$. In this case, by taking into account the $\delta_{\lambda}%
$-homogeneity of the $X_{i}$'s, we are entitled to invoke the results in
\cite{BBConstruction, BonfLanc}: there exists an o\-pe\-ra\-tion $\circ$ on
$\mathbb{R}^{n}$ such that $\mathbb{F}:=(\mathbb{R}^{n},\circ,\delta_{\lambda
})$ is a Carnot group, and
\[
\mathrm{Lie}(\mathbb{F})=\mathfrak{a}.
\]
As a consequence, $X_{1},\ldots,X_{m}$ are homogeneous and left-invariant on
$\mathbb{F}$, and thus the results in this paper are well-known (see
\cite[Thm.\,2.5]{BLUpaper}). \vspace*{0.15cm}

\item[(b)] $N>n$. In this case, again by exploiting the results contained in
\cite{BBConstruction}, we see that \emph{there cannot exist} any Lie-group
structure in $\mathbb{R}^{n}$ with respect to which $X_{1},\ldots,X_{m}$ are
left-invariant. In particular, \cite[Thm.\,2.5]{BLUpaper} \emph{does not
apply} in this case, and analogous results are not known.
\end{itemize}

In view of the preceding discussion, it is not restrictive to assume the
following `dimensional' hypothesis (in addition to (H.1) and (H.2)).

\begin{description}
\item[(H.3)] We suppose that
\[
N=\mathrm{dim}(\mathfrak{a})>n,
\]
and we define
\[
p:=N-n\geq1.
\]

\end{description}

\begin{remark}
\label{Remark H3} All the results of this paper will be stated assuming only
(H.1) and (H.2); however, their proofs will be given assuming also (H.3). The
reason is that, if (H.1) and (H.2) hold but (H.3) is not satisfied, that is,
$N = n$, all these results are already known from \cite{BLUpap2} and
\cite{BoUg}, as explained in the Introduction.
\end{remark}

We will denote points $z\in\mathbb{R}^{N}$ by%
\begin{equation}
\label{eq.notationz}z=(x,\xi),\qquad\text{with $x\in\mathbb{R}^{n}$ and
$\xi\in\mathbb{R}^{p}$}.
\end{equation}
Under assumption (H.3) it is proved in \cite{BiagiBonfiglioliLast} that the
$X_{i}$'s can be lifted (in a suitable sense) to \emph{left-invariant} vector
fields `living' on a higher-dimensional Carnot group:

\begin{theorem}
[{{see \cite[Thm.\,3.2]{BiagiBonfiglioliLast}}}]\label{thm.liftingGroupG}
Assume that $X=\{X_{1},\ldots,X_{m}\}\subseteq\mathcal{X}(\mathbb{R}^{n})$
sa\-ti\-sfies \emph{(H.1)}, \emph{(H.2)}, \emph{(H.3)}. \vspace*{0.05cm}

Then, there exists a homogeneous Carnot group $\mathbb{G}=(\mathbb{R}^{N}%
,\ast,D_{\lambda})$, nilpotent of step $r=\sigma_{n}$ and with $m$ generators,
such that
\[
\text{$\mathrm{Lie}(\mathbb{G})$ is isomorphic to $\mathfrak{a}$}.
\]
Moreover, using the notation in \eqref{eq.notationz}, the dilation
$D_{\lambda}$ takes the `lifted form'
\begin{equation}
\label{eq.Dlambdalifted}D_{\lambda}(x,\xi)=(\delta_{\lambda}(x),\lambda
^{s_{1}}\xi_{1},\ldots,\lambda^{s_{p}}\xi_{p}), \quad%
\begin{array}
[c]{c}%
\text{where $s_{1},\ldots,s_{p}\in\mathbb{N}$ and}\\
s_{1}\leq\ldots\leq s_{p}<\sigma_{n}.
\end{array}
\end{equation}
As a consequence, the homogeneous dimension $Q$ of $\mathbb{G}$ is given by
\begin{equation}
\label{eq.homogQG}\textstyle Q:=\sum_{i=1}^{n}\sigma_{i}+\sum_{i=1}^{p}%
s_{i}>q.
\end{equation}
Finally, there exists a system $\widehat{X}=\{\widehat{X}_{1},\ldots,
\widehat{X}_{m}\}$ of Lie-generators of $\mathrm{Lie}(\mathbb{G})$ such that
$\widehat{X}_{i}$ is a lifting of $X_{i}$ for every $i=1,\ldots,m$; this means
that
\begin{equation}
\label{lifting}\widehat{X}_{i}(x,\xi)=X_{i}(x)+R_{i}(x,\xi),
\end{equation}
where $R_{i}(x,\xi)$ is a smooth vector field operating only in the variable
$\xi\in\mathbb{R}^{p}$, with coefficients possibly depending on $(x,\xi)$. In
particular, the $\widehat{X}_{i}$'s are $D_{\lambda}$-ho\-mo\-ge\-neous of
degree $1$ \emph{(}for every $i=1,\ldots,m$\emph{)}.
\end{theorem}

\begin{remark}
\label{rem.reviewConst} For a future reference, here we briefly review how the
group $\mathbb{G}$ in Theorem \ref{thm.liftingGroupG} is constructed. For all
the details, we refer to \cite{BiagiBonfiglioliLast}. \vspace*{0.05cm}

First of all, since we have already recognized that $\mathfrak{a}$ is
nilpotent and stratified, it is well-known that $(\mathfrak{a},\diamond
,\Delta_{\lambda})$ is a stratified group, where

\begin{itemize}
\item $\diamond$ is the Baker-Campbell-Hausdorff series on $\mathfrak{a}$
(boiling down to a finite sum, since $\mathfrak{a}$ is nilpotent);
\vspace*{0.1cm}

\item $\Delta_{\lambda}$ is the unique linear map on $\mathfrak{a}$ such that
$\Delta_{\lambda}|_{V_{i}}:=\lambda^{i}\,\mathrm{id}$.
\end{itemize}

Moreover, since $\mathfrak{a}$ has finite dimension $N$, we can fix a basis
\[
\mathcal{E}=\{E_{1},\ldots,E_{N}\}
\]
of $\mathfrak{a}$ (as a vector space), which is \emph{adapted to the
stratification} $\{V_{i}\}_{i=1}^{\sigma_{n}}$ in \eqref{eq.LieXoplus}. This
means that, setting $r:=\sigma_{n}$, $\mathcal{E}$ can be decomposed as
\[
\mathcal{E}=\left\{  E_{1}^{(1)},\ldots E_{N_{1}}^{(1)},\ldots,E_{1}%
^{(r)},\ldots,E_{N_{r}}^{(r)}\right\}  ,
\]
where, for every $i=1,\ldots,r$, we have

\begin{itemize}
\item $N_{i} := \mathrm{dim}(V_{i})$ (so that $N_{1} = m$ and $N_{1}%
+\cdots+N_{r} = N$); \vspace*{0.1cm}

\item $\mathcal{E}_{i}:=\left\{  E_{1}^{(i)},\ldots,E_{N_{i}}^{(i)}\right\}  $
is a basis of $V_{i}$.
\end{itemize}

Using the chosen basis $\mathcal{E}$, we then equip $\mathbb{R}^{N}$ with a
structure of homogeneous Carnot group $\mathbb{A}=(\mathbb{R}^{N}%
,\circ,d_{\lambda})$ by `reading' $\diamond$ and $\Delta_{\lambda}$ in
$\mathcal{E}$-coordinates, i.e.,
\begin{align*}
&  \,\sum_{i=1}^{N}(a\circ b)_{i}\,E_{i}=\left(  \sum_{i=1}^{N}a_{i}%
\,E_{i}\right)  \diamond\left(  \sum_{i=1}^{N}b_{i}\,E_{i}\right)
\qquad(\text{for all $a,b\in\mathbb{R}^{N}$})\\[0.15cm]
&  \,\sum_{i=1}^{N}\left(  d_{\lambda}(a)\right)  _{i}\,E_{i}= \Delta
_{\lambda}\left(  \sum_{i=1}^{N}a_{i}\,E_{i}\right)  \qquad(\text{for all
$a\in\mathbb{R}^{N}$ and $\lambda>0$}).
\end{align*}
For any fixed $i=1,\ldots,m$, we now let $J_{i}$ be the unique left-invariant
vector field on $\mathbb{A}$ coinciding with $\partial_{a_{i}}$ at $a=0$. In
\cite{BiagiBonfiglioliLast} it is proved the existence of a suitable
diffeomorphism $T\in C^{\infty}(\mathbb{R}^{N};\mathbb{R}^{N})$, only
depending on the basis $\mathcal{E}$, which turns the $J_{i}$'s into new
vector fields, say $Z_{1},\ldots,Z_{m}\in\mathcal{X}(\mathbb{R}^{N})$, such
that
\[
Z_{i}(z)=Z_{i}(x,\xi)=E_{i}(x)+W_{i}(x,\xi)\qquad(i=1,\ldots,m).
\]
Here, $W_{1},\ldots,W_{m}$ are smooth vector fields operating only in the
variable $\xi\in\mathbb{R}^{p}$, with coefficients possibly depending on
$(x,\xi)$. On the other hand, since $\mathcal{E}_{1}$ is a basis of
$V_{1}=\mathrm{span}\{X_{1},\ldots,X_{m}\}$, for every $i=1,\ldots,m$ we can
write
\[
\textstyle X_{i}=\sum_{k=1}^{m}c_{k,i}E_{k}\qquad\text{(for a suitable
constants $c_{k,i}\in$}\mathbb{R}\text{)}.
\]
Hence, the set $\widehat{X}$ is obtained by defining
\[
\textstyle \widehat{X}_{i}:=\sum_{k=1}^{m}c_{k,i}Z_{k}\qquad(\text{for all
$i=1,\ldots,m$}).
\]
Finally, the underlying Carnot group $\mathbb{G}$ appearing in the statement
of Theorem \ref{thm.liftingGroupG} can be obtained as the unique Lie group
isomorphic to $\mathbb{A}$ via $T$.
\end{remark}

We close this section by recalling the notion of \emph{control distance}
associated with a H\"{o}rmander set of vector fields; moreover,  we review and
some  properties of this distance  when \emph{homogeneous vector fields} are involved.

\begin{definition}
\label{def.CCdistance} Let $\mathcal{W}=\{W_{1},\ldots,W_{m}\}$ be a family of
smooth vector fields satisfying H\"{o}rmander's rank condition at every point
of $\mathbb{R}^{n}$. Given any couple of points $x,y\in\mathbb{R}^{n}$, we
define
\[
d_{\mathcal{W}}(x,y):=\inf\left\{  T>0:\,\text{there exists $\gamma
\in\mathcal{S}(T)$ with $\gamma(0)=$x and $\gamma(T)=$}y\right\}  ,
\]
where $\mathcal{S}(T)$ is the set of the $W^{1,1}$-curves $\gamma
:[0,T]\rightarrow\mathbb{R}^{n}$ satisfying
\[
\dot{\gamma}(t)=\sum_{j=1}^{m}\alpha_{j}(t)\,W_{j}(\gamma(t)),\quad\text{with
$\sum_{j=1}^{m}|\alpha_{j}(t)|\leq1$}.
\]
The map $d_{\mathcal{W}}$ is called the $\mathcal{W}$-control di\-stan\-ce or
the Carnot-Carath\'{e}odory di\-stan\-ce (CC distance) related to
$\mathcal{W}$. Given any $x\in\mathbb{R}^{n}$ and a\-ny $r>0$, we indicate by
$B_{\mathcal{W}}(x,r)$ the $d_{\mathcal{W}}$-ball
\[
B_{\mathcal{W}}(x,r):=\left\{  y\in\mathbb{R}^{n}:\,d_{\mathcal{W}%
}(x,y)<r\right\}  .
\]

\end{definition}

\begin{remark}
\label{rem.continuitysame} Since we have assumed that $W_{1},\ldots,W_{m}$
sa\-ti\-sfy H\"{o}rmander's rank con\-di\-tion at every point of
$\mathbb{R}^{n}$, it is well-known that $d_{\mathcal{W}}(x,y)$ is finite for
every $x,y\in\mathbb{R}^{n}$; moreover, $d_{\mathcal{W}}$ is a distance,
topologically but not metrically equivalent to the Euclidean one. In
particular, \textquotedblleft con\-ti\-nuo\-us\textquotedblright\ in Euclidean
sense and \textquotedblleft con\-ti\-nuo\-us\textquotedblright\ with respect
to the control distance $d_{\mathcal{W}}$ are the same.
\end{remark}

\begin{remark}
\label{rem.dXhomog} Let $\mathcal{W}=\{W_{1},\ldots,W_{m}\}\subseteq
\mathcal{X}(\mathbb{R}^{n})$ satisfy assumptions (H.1)-(H.2), and let
$d_{\mathcal{W}}$ be the associated control distance. Due to the
$\delta_{\lambda}$-homogeneity of the $W_{i}$'s, it easy to see that the
following properties hold:

\begin{enumerate}
\item $d_{\mathcal{W}}$ is jointly $\delta_{\lambda}$-homogeneous of degree
$1$, that is,
\[
d_{\mathcal{W}}(\delta_{\lambda}(x),\delta_{\lambda}(y))=\lambda
\,d_{\mathcal{W}}(x,y)\quad\text{for all $x,y\in\mathbb{R}^{n}$ and
$\lambda>0$};
\]

\item for every $x\in\mathbb{R}^{n}$ and every $r>0$, one has
\[
\delta_{\lambda}\left(  B_{\mathcal{W}}(x,r)\right)  =B_{\mathcal{W}}%
(\delta_{\lambda}(x),\lambda r).
\]

\end{enumerate}

If, in addition $W_{1},\ldots,W_{m}$ are left-invariant with respect to some
Lie-group stru\-ctu\-re $\mathbb{F}=(\mathbb{R}^{n},\circ)$, the distance
$d_{\mathcal{W}}$ is also \emph{translation-invariant}, that is,
\[
d_{\mathcal{W}}(x,y)=d_{\mathcal{W}}(\alpha\ast x,\alpha\ast y)\quad\text{for
all $x,y,\alpha\in\mathbb{R}^{n}$}.
\]
In particular, the above property implies that
\[
d_{\mathcal{W}}(0,x)=d_{\mathcal{W}}(0,x^{-1})\quad\text{for all
$x\in\mathbb{R}^{m}$}.
\]

\end{remark}

\begin{remark}
\label{rem.liftingdXdhatX} Let $X=\{X_{1},\ldots,X_{m}\}\subseteq
\mathcal{X}(\mathbb{R}^{n})$ satisfy (H.1), (H.2), (H.3), and let
\[
\mathbb{G}=(\mathbb{R}^{N},\ast,D_{\lambda}),\qquad\widehat{X}:=\{\widehat{X}%
_{1},\ldots,\widehat{X}_{m}\}
\]
be as in Theorem \ref{thm.liftingGroupG}. Moreover, let $d_{X}$ and
$d_{\widehat{X}}$ denote the CC-distances associated with $X$ and
$\widehat{X}$, respectively. Since, for every $j=1,\ldots,m$, we have
\[
\textstyle \widehat{X}_{j}=X_{j}+\sum_{i=1}^{p}r_{i,j}(x,\xi)\,\partial
_{\xi_{i}}%
\]
(for suitable smooth functions $r_{i,j}$), it is easy to recognize that
\[
d_{X}(x,y)\leq d_{\widehat{X}}\left(  (x,\xi),(y,\eta)\right)  \qquad
\forall\,\,x,y\in\mathbb{R}^{n},\,\xi,\eta\in\mathbb{R}^{p}.
\]
Furthermore, given any $z=(x,\xi)\in\mathbb{R}^{N}$ and any $r>0$, we have
\begin{equation}
\pi\left(  B_{\widehat{X}}(z,r)\right)  =B_{X}(x,r),
\label{eq.projectionBalls}%
\end{equation}
where $\pi:\mathbb{R}^{N}\equiv\mathbb{R}^{n}\times\mathbb{R}^{p}%
\rightarrow\mathbb{R}^{n}$ denotes the projection of $\mathbb{R}^{N}$ onto
$\mathbb{R}^{n}$. We explicitly notice that, since $\pi$ is continuous, from
\eqref{eq.projectionBalls} we immediately derive that
\[
\pi\left(  \overline{B_{\widehat{X}}(z,r)}\right)  \subseteq\overline
{B_{X}(x,r)}.
\]

\end{remark}

\section{Uniform Gaussian bounds for operators with constant coefficients}

\label{sec.constantCoeff} Given $\Lambda\geq1$, we denote by $\mathcal{M}%
_{\Lambda}$ the set of the $m\times m$ symmetric matrices $A$ satisfying the
following \emph{uniform ellipticity condition}:
\[
\frac{1}{\Lambda}|\xi|^{2}\leq\langle A\xi,\xi\rangle\leq\Lambda|\xi
|^{2}\qquad\text{for every $\xi\in\mathbb{R}^{m}$}.
\]
For every fixed $A=(a_{i,j})_{i,j=1}^{m}\in\mathcal{M}_{\lambda}$ we define%
\begin{equation}
\label{eq.defiOpLLA}\mathcal{H}_{A}:=\sum_{i,j=1}^{m}a_{i,j}X_{i}%
X_{j}-\partial_{t}\qquad\text{on $\mathbb{R}^{1+n}=\mathbb{R}_{t}%
\times\mathbb{R}_{x}^{n}$},
\end{equation}
Since $A$ is symmetric and positive definite, it admits a \emph{unique}
(symmetric and) positive definite square root, say $S$. As a
con\-se\-quen\-ce, writing $S=(s_{i,j})_{i,j=1}^{m}$, we have
\[
\mathcal{H}_{A}=\sum_{j=1}^{m}Y_{j}^{2}-\partial_{t},\qquad\text{where
$Y_{j}:=\sum_{i=1}^{m}s_{i,j}X_{i}$}.
\]
On the other hand, since $S$ is non-singular, the family $Y=\{Y_{1}%
,\ldots,Y_{m}\}\subseteq\mathcal{X}(\mathbb{R}^{n})$ satisfies assumptions
(H.1)-(H.2)-(H.3). In particular, since the $s_{i,j}$'s are constant, one has
\begin{equation}
\label{eq.LieYa}\mathrm{Lie}(Y)=\mathrm{Lie}(Y_{1},\ldots,Y_{m}) =
\mathfrak{a}.
\end{equation}
Gathering these facts, we can apply Theorem \ref{thm.liftingGroupG} to the
family $Y$: there exist a homogeneous Carnot group $\mathbb{F}$ and a system
\[
\widehat{Y}=\{\widehat{Y}_{1},\ldots,\widehat{Y}_{m}\}
\]
of Lie-generators for the Lie algebra $\mathrm{Lie}(\mathbb{F})$ such that,
for every $i=1,\ldots,m$, the vector field $\widehat{Y}_{i}$ is a lifting of
$Y_{i}$ in the sense of \eqref{lifting}. \vspace*{0.05cm}

The key observation is that, in view of Remark \ref{rem.reviewConst}, the
construction of $\mathbb{F}$ does not really depend on $Y$, but only on the
Lie algebra $\mathrm{Lie}(Y)$ and on the choice of an \emph{adapted basis}. As
a consequence, using \eqref{eq.LieYa} and choosing the \emph{same adapted
basis} used for the construction of $\mathbb{G}$ (notice that, by
\eqref{eq.Viintrinsic}, the stra\-ti\-fi\-ca\-tion $\{V_{k}\}_{k=1}%
^{\sigma_{n}}$ is independent of $X$ or $Y$), we obtain
\[
\mathbb{F}=\mathbb{G}\qquad\text{and}\qquad\widehat{Y}_{j}=\sum_{i=1}%
^{m}s_{i,j}\widehat{X}_{i}.
\]
Summing up, the couple $(\mathbb{G},\widehat{X})$ associated with $X$ provides
a `lifting pair' for the family $Y=\{Y_{1},\ldots,Y_{m}\}$ (in the sense of
Theorem \ref{thm.liftingGroupG}) \emph{which is independent of the fixed
matrix $A\in\mathcal{M}_{\Lambda}$.} By making use of this observation, we are
entitled to use the results established in \cite{BBHeat}, which lead to the
following theorem.

\begin{theorem}
[{{see \cite[Thm.s\,1.4 and 1.6]{BBHeat}}}]\label{thm.existenceGammaprop} With
the above assumptions and no\-ta\-tion, the following facts hold.

\begin{enumerate}
\item If $\widehat{\Gamma}_{A}$ is the \emph{(}unique\emph{)} smooth heat
kernel of
\[
\widehat{\mathcal{H}}_{A}=\sum_{i,j=1}^{m}a_{i,j}\widehat{X}_{i}
\widehat{X}_{j}-\partial_{t}=\sum_{i=1}^{m} \widehat{Y}_{i}^{2}-\partial
_{t}\qquad\text{on $\mathbb{R}_{t}\times\mathbb{R}_{(x,\xi)}^{N}$}%
\]
vanishing at infinity con\-struc\-ted in \cite{Fo}, then the function
\begin{equation}
\label{sec.one:mainThm_defGamma}%
\begin{split}
\Gamma_{A}  &  : \mathbb{R}^{1+n}\times\mathbb{R}^{1+n}\rightarrow
\mathbb{R},\\
\Gamma_{A}\left(  t,x;s,y\right)   &  = \int_{\mathbb{R}^{p}}\widehat{\Gamma
}_{A} \left(  t,\left(  x,\xi\right)  ;s,\left(  y,0\right)  \right)  d\xi,
\end{split}
\end{equation}
is a \emph{global heat kernel} for the operator $\mathcal{H}_{A}$ defined in
\eqref{eq.defiOpLLA}. This means, precisely, that

\begin{itemize}
\item for any fixed $(t,x)\in\mathbb{R}^{1+n}$, we have $\Gamma_{A}%
(t,x;\cdot)\in L_{\mathrm{loc}}^{1}(\mathbb{R}^{1+n})$;

\item for every $\varphi\in C_{0}^{\infty}(\mathbb{R}^{1+n})$ and every
$(t,x)\in\mathbb{R}^{1+n}$, we have
\[%
\begin{split}
&  \mathcal{H}_{A}\left(  \int_{\mathbb{R}^{1+n}}\Gamma_{A}(t,x;s,y)\varphi
(s,y)dsdy\right) \\
&  =\int_{\mathbb{R}^{1+n}}\Gamma_{A}(t,x;s,y)\,\mathcal{H}_{A}\varphi
(s,y)dsdy= -\varphi(t,x).
\end{split}
\]

\end{itemize}

Moreover, $\Gamma_{A}$ enjoys the properties listed below. \vspace*{0.05cm}

\begin{enumerate}
\item[(a)] $\Gamma_{A}\geq0$ and $\Gamma_{A}(t,x;s,y)=0$ if and only if $t\leq
s$; \vspace*{0.07cm}

\item[(b)] $\Gamma_{A}(t,x;s,y)=\Gamma_{A}(t,y;s,x)$\,\,and\,\,$\Gamma
_{A}(t,x;s,y)=\Gamma_{A}(t-s,x;0,y)$;

\item[(c)] $\Gamma_{A}$ is smooth out of the diagonal of $\mathbb{R}%
^{1+n}\times\mathbb{R}^{1+n}$, and
\[
\mathcal{H}_{A}\left(  \Gamma_{A}(\cdot;s,y)\right)  \equiv0\quad\text{on
$\mathbb{R}^{1+n}\setminus\{(s,y)\}$};
\]

\item[(d)] for every fixed $(t,x)\in\mathbb{R}^{1+n}$, if $t>s$ we have
\[
\int_{\mathbb{R}^{n}}\Gamma_{A}(t,x;s,y)dy= \int_{\mathbb{R}^{n}}\Gamma
_{A}(t,x;s,y)dx=1;
\]

\item[(e)] for every $x,y\in\mathbb{R}^{n}$ and $s<\tau<t$ we have
\begin{equation}
\label{reproduction prop Gamma cost}\Gamma_{A}(t,x;s,y)= \int_{\mathbb{R}^{n}%
}\Gamma_{A}(t,x;\tau,\zeta)\,\Gamma_{A}(\tau,\zeta;s,y)d\zeta.
\end{equation}

\end{enumerate}

\item Setting $\widehat{\gamma}_{A}(t,z):=\widehat{\Gamma}_{A}(t,z;0)$, we
have
\[
\widehat{\Gamma}_{A}\left(  t,(x,\xi);s,(y,\eta)\right)  = \widehat{\gamma
}_{A}\left(  t-s,(y,\eta)^{-1}\ast(x,\xi)\right)  ,
\]
so that identity \eqref{sec.one:mainThm_defGamma} becomes
\begin{equation}
\label{sec.one:mainThm_defGamma22222}\Gamma_{A}(t,x;s,y)=\int_{\mathbb{R}^{p}%
}\widehat{\gamma}_{A}\left(  t-s,(y,0)^{-1}\ast(x,\xi)\right)  d\xi.
\end{equation}

\item For every integers $h,k\geq1$, $\alpha,\beta\geq0$ and every choice of
$i_{1},\ldots,i_{h},j_{1},\ldots,j_{k}$ in $\{1,\ldots,m\}$, we have the
following representation formulas, holding true for every $x,y\in
\mathbb{R}^{n}$ and $t,s\in\mathbb{R}$ such that $s<t$:
\begin{align}
&  (\partial_{t})^{\alpha}(\partial_{s})^{\beta}\,X_{i_{1}}^{x}\cdots
X_{i_{h}}^{x}\Gamma_{A}(t,x;s,y)\label{eq.derGammatx}\\[0.15cm]
&  =(-1)^{\beta}\int_{\mathbb{R}^{p}}\left(  (\partial_{t})^{\alpha+\beta}
\widehat{X}_{i_{1}}\cdots\widehat{X}_{i_{h}}\widehat{\gamma}_{A}\right)
\left(  t-s,(y,0)^{-1}\ast(x,\xi)\right)  d\xi;\nonumber\\[0.3cm]
&  (\partial_{t})^{\alpha}(\partial_{s})^{\beta}\,X_{j_{1}}^{y}\cdots
X_{j_{k}}^{y}\Gamma_{A}(t,x;s,y)\label{eq.derGammays}\\[0.15cm]
&  =(-1)^{\beta}\,\int_{\mathbb{R}^{p}}\left(  (\partial_{t})^{\alpha+\beta}
\widehat{X}_{j_{1}}\cdots\widehat{X}_{j_{k}}\widehat{\gamma}_{A}\right)
\left(  t-s,(x,0)^{-1}\ast(y,\xi)\right)  d\xi;\nonumber\\[0.3cm]
&  (\partial_{t})^{\alpha}(\partial_{s})^{\beta}\,X_{j_{1}}^{y}\cdots
X_{j_{k}}^{y}X_{i_{1}}^{x}\cdots X_{i_{h}}^{x}\Gamma_{A}%
(t,x;s,y)\label{eq.derGammatutte}\\[0.15cm]
&  =(-1)^{\beta}\,\int_{\mathbb{R}^{p}}\left(  (\partial_{t})^{\alpha+\beta}
\widehat{X}_{j_{1}}\cdots\widehat{X}_{j_{k}}\left(  (\widehat{X}_{i_{1}}
\cdots\widehat{X}_{i_{h}}\widehat{\gamma}_{A})\circ\widetilde{\iota}\right)
\right) \nonumber\\
&  \qquad\qquad\qquad\left(  t-s,(x,0)^{-1}\ast(y,\xi)\right)  \,d\xi
\,,\nonumber
\end{align}
Here $\widetilde{\iota}:\mathbb{R}^{1+N}\rightarrow\mathbb{R}^{1+N}$ is the
map defined by
\begin{equation}
\label{eq.iotaiotatilde}\widetilde{\iota}(t,(x,\xi))=(t,\iota(x,\xi
))\qquad(\text{with $t\in\mathbb{R}$, $x\in\mathbb{R}^{n}$, $\xi\in
\mathbb{R}^{p}$}),
\end{equation}
and $\iota(x,\xi)=(x,\xi)^{-1}$ is the inverse of $(x,\xi)$ in the Carnot
group $\mathbb{G}$.
\end{enumerate}
\end{theorem}

\begin{remark}
\label{rem.representationSymm} By combining the representation formula
\eqref{sec.one:mainThm_defGamma22222} with the symmetry of $\Gamma_{A}$ in $x$
and $y$, we obtain the following alternative identity
\begin{equation}
\label{eq.reprGammaASymm}\Gamma_{A}(t,x;s,y)=\Gamma_{A}(t,y;s,x)=\int%
_{\mathbb{R}^{p}}\widehat{\gamma}_{A} \left(  t-s,(x,0)^{-1}\ast
(y,\xi)\right)  d\xi.
\end{equation}
We shall repeatedly exploit \eqref{eq.reprGammaASymm} in place of \eqref{sec.one:mainThm_defGamma22222}.
\end{remark}

By means of the representation formula \eqref{sec.one:mainThm_defGamma22222}
for $\Gamma_{A}$ and of the analogous representation formulas
\eqref{eq.derGammatx}-\eqref{eq.derGammatutte} for its $(t,X)$-derivatives, we
are able to prove the following theorem, which is the main result in this
section. This will be the starting point to implement the parametrix method
and build a fundamental solution for operators with variable coefficients.

\begin{theorem}
[Gaussian bounds for constant coefficient operators]%
\label{thm.mainParabolicCase} Let the above assumptions and notation do apply.
Moreover, let $\Lambda\geq1$ be fixed. Then, the following facts hold.

\begin{enumerate}
\item There exists a constant $\kappa_{\Lambda}>0$ such that
\begin{equation}
\label{eq.uniformGammaPar}%
\begin{split}
&  \frac{1}{\kappa_{\Lambda}}\,\frac{1}{\left\vert B_{X}(x,\sqrt
{t-s})\right\vert }\,\exp\left(  -\kappa_{\Lambda}\frac{d_{X}(x,y)^{2}}
{t-s}\right)  \leq\Gamma_{A}(t,x;s,y)\\
&  \leq\kappa_{\Lambda}\,\frac{1}{\left\vert B_{X}(x,\sqrt{t-s})\right\vert }
\,\exp\left(  -\frac{d_{X}(x,y)^{2}}{\kappa_{\Lambda}(t-s)}\right)  ,
\end{split}
\end{equation}
for every $x,y\in\mathbb{R}^{n}$, $s<t$ and $A\in\mathcal{M}_{\Lambda}$.

\item For every integers $r\geq1,\alpha,\beta\geq0$, there exists
$\kappa=\kappa_{\Lambda,r,\alpha,\beta}>0$ such that
\begin{equation}
\label{eq.uniformDerGammaPar}%
\begin{split}
&  \left\vert (\partial_{t})^{\alpha}\,(\partial_{s})^{\beta}\,W_{1}\cdots
W_{r}\Gamma_{A}(t,x;s,y)\right\vert \\
&  \leq\kappa\,\frac{(t-s)^{-(\alpha+\beta+r/2)}}{\left\vert B_{X}
(x,\sqrt{t-s})\right\vert }\,\exp\left(  -\frac{d_{X}(x,y)^{2}}{\kappa
(t-s)}\right)  ,
\end{split}
\end{equation}
for every $x,y\in\mathbb{R}^{n}$, $s<t$, \emph{$A\in\mathcal{M}_{\Lambda}$}
and every choice of $W_{1},\ldots,W_{r}$ in
\[
\mathcal{D}_{X}:=\left\{  X_{1}^{x},\ldots,X_{m}^{x},X_{1}^{y},\ldots
,X_{m}^{y}\right\}  .
\]

\item For every integers $r\geq1,\alpha,\beta\geq0$, there exists
$\overline{\kappa}=\overline{\kappa}_{\Lambda,r,\alpha,\beta}>0$ such that
\begin{equation}
\label{eq.uniformGammaABPar}%
\begin{split}
&  \left\vert \left(  \partial_{t}\right)  ^{\alpha}\left(  \partial
_{s}\right)  ^{\beta}W_{1}...W_{r} \Gamma_{A}\left(  t,x;s,y\right)  -\left(
\partial_{t}\right)  ^{\alpha}\left(  \partial_{s}\right)  ^{\beta}%
W_{1}...W_{r} \Gamma_{B}\left(  t,x;s,y\right)  \right\vert \\
&  \leq\overline{\kappa}\left\Vert A-B\right\Vert ^{1/\sigma_{n}}
\frac{\left(  t-s\right)  ^{-\left(  \alpha+\beta+r/2\right)  }} {\left\vert
B_{X}\left(  x,\sqrt{t-s}\right)  \right\vert } \exp\left(  -\frac
{d_{X}\left(  x,y\right)  ^{2}}{\kappa\left(  t-s\right)  }\right)  ,
\end{split}
\end{equation}
for every $x,y\in\mathbb{R}^{n}$, $s<t$, \emph{$A,B\in\mathcal{M}_{\Lambda}$}
and every $W_{1},\ldots,W_{r}\in\mathcal{D}_{X}$ \emph{(}here, $\|\cdot\|$
denotes the usual matrix norm\emph{)}.
\end{enumerate}
\end{theorem}

In order to establish Theorem \ref{thm.mainParabolicCase} we need the following

\begin{lemma}
\label{lem.technicVF} Let $Z\in\mathcal{X}(\mathbb{R}^{N})$ be $D_{\lambda}%
$-homogeneous of degree $1$. Then,
\begin{equation}
\label{eq.actionZfLemma}Z=\sum_{i=1}^{\sigma_{n}}\gamma_{i}(z)\mathcal{P}%
_{i}(\widehat{X}_{1}, \ldots,\widehat{X}_{m}),
\end{equation}
where $\mathcal{P}_{i}(\theta_{1},\ldots,\theta_{m})$ is a suitable
homogeneous polynomial of \emph{(}Euclidean\emph{)} degree $i$ in the
non-commuting variables $\theta_{1},\ldots,\theta_{m}$, and $\gamma_{i}$ is a
$D_{\lambda}$-homogeneous polynomial of degree $i-1$ \emph{(}for all
$i=1,\ldots,\sigma_{n}$\emph{)}.
\end{lemma}

\begin{proof}
Throughout this proof we do not need to separate the `base' va\-ri\-able
$x\in\mathbb{R}^{n}$ from the `lifted' variable $\xi\in\mathbb{R}^{p}$; hence
we use the com\-pact notation $z=(z_{1},\ldots,z_{N})$ for the points of
$\mathbb{R}^{N}$ and we write
\[
D_{\lambda}(z)=\left(  \lambda^{\upsilon_{1}}z_{1},\ldots,\lambda
^{\upsilon_{N}}z_{N}\right)  .
\]
For every fixed $i\in\{1,\ldots,N\}$, we then denote by $J_{i}$ the unique
left-invariant vector field on $\mathbb{G}$ coinciding with $\partial_{z_{i}}$
at $z=0$. By well-known results on Carnot groups (see, e.g., \cite[Sec.s 1.3
and 1.4]{BLUlibro}), $J_{i}$ is $D_{\lambda}$-homogeneous of degree
$\upsilon_{i}$ and
\begin{equation}
\label{eq.exprJi}J_{i}=\partial_{z_{i}}+\sum_{\underset{\upsilon_{k}%
>\upsilon_{i}}{k=1}}^{N} \alpha_{k,i}(z)\,\frac{\partial}{\partial z_{k}},
\end{equation}
where $\alpha_{k,i}$ is a suitable $D_{\lambda}$-homogeneous polynomial of
degree $\upsilon_{k}-\upsilon_{i}$. Starting form \eqref{eq.exprJi}, it is
easy to recognize that
\[
\partial_{z_{i}}=J_{i}+\sum_{\underset{\upsilon_{k}>\upsilon_{i}}{k=1}}%
^{N}\beta_{k,i}(z)\,J_{k},
\]
where $\beta_{k,i}$ is again a $D_{\lambda}$-homogeneous polynomial of degree
$\upsilon_{k}-\upsilon_{i}$. Thus, if $Z$ is any smooth vector field
$D_{\lambda}$-homogeneous of degree $1$, we can write
\begin{equation}
\label{eq.formZLemma}Z=\sum_{k=1}^{N}a_{k}(z)\,\frac{\partial}{\partial z_{k}%
}= \sum_{k=1}^{N}\gamma_{k}(z)\,J_{k},
\end{equation}
where $\gamma_{k}$ is a $D_{\lambda}$-homogeneous polynomial of degree
$\upsilon_{k}-1$. Now, since $J_{1},\ldots,J_{N}$ are left-invariant on
$\mathbb{G}$ and $\widehat{X}_{1},\ldots,\widehat{X}_{m}$ are
\emph{Lie-generators} of $\mathrm{Lie}(\mathbb{G})$, any $J_{i}$ can be
written as a linear combination (with constant coefficients) of iterated
commutators of length $\upsilon_{i}$ of the $\widehat{X}_{i}$'s; more
precisely, we have
\begin{equation}
\label{eq.JipolPi}J_{i}=\mathcal{P}_{i}(\widehat{X}_{1},\ldots,\widehat{X}%
_{m})\qquad(i=1,\ldots,N),
\end{equation}
where $\mathcal{P}_{i}(\theta_{1},\ldots,\theta_{m})$ is a suitable
homogeneous polynomial of (Euclidean) degree $\upsilon_{i}$ in the
non-commuting variable $\theta_{1},\ldots,\theta_{m}$. \vspace*{0.15cm}

Combining \eqref{eq.formZLemma} and \eqref{eq.JipolPi}, we get
\[
Z=\sum_{k=1}^{N}\gamma_{k}(z)\mathcal{P}_{k}(\widehat{X}_{1},\ldots
,\widehat{X}_{m}).
\]
Finally, reminding that $\min_{k}\upsilon_{k}=1$ and $\max_{k}\upsilon
_{k}=\sigma_{n}$ (see, respectively, \eqref{eq.defidela} and
\eqref{eq.Dlambdalifted}), we can reorder the above sum with respect the
$D_{\lambda}$-homogeneity of the $\mathcal{P}_{k}$'s, thus obtaining \eqref{eq.actionZfLemma}.
\end{proof}

We will also need the next

\begin{proposition}
[{See \cite[Prop. 3.10, Rem. 3.9]{BiBraGauss}}]\label{Prop global doubling}
The following \emph{global doubling property} of $d_{X}$ holds: \emph{there
exist $\gamma_{1},\gamma_{2}>0$ such that}
\begin{equation}
\label{eq.globalD}\gamma_{1}\left(  \frac{r}{\rho}\right)  ^{n}\leq
\frac{|B_{X}(x,r)|} {|B_{X}(x,\rho)|}\leq\gamma_{2}\left(  \frac{r}{\rho
}\right)  ^{q},
\end{equation}
for every $x\in\mathbb{R}^{n}$ and every $0<\rho<r$. This also implies, for
every $\theta> 0$,
\begin{equation}
\label{eq.symmHxy}\frac{1}{\left\vert B_{X}(y,\sqrt{r})\right\vert }%
\,\exp\left(  -\frac{d_{X}(x,y)^{2}}{\theta\,r}\right)  \leq\frac{C_{q}%
}{\left\vert B_{X} (x,\sqrt{r})\right\vert }\,\exp\left(  -\frac
{d_{X}(x,y)^{2}}{C_{q}\,\theta\,r}\right)  ,
\end{equation}
where $C_{q}>0$ is a constant only depending on the number $q$ in \eqref{eq.defqhom}.
\end{proposition}

We can now prove Theorem \ref{thm.mainParabolicCase}.

\begin{proof}
[Proof of Theorem \ref{thm.mainParabolicCase}](1)\thinspace\thinspace First of
all, since $\mathbb{G}$ is a Carnot group and $\widehat{X}_{1},\ldots
,\widehat{X}_{m}$ are Lie-generators of $\mathrm{Lie} (\mathbb{G})$, we can
apply \cite[Thm.\,2.5]{BLUpaper}: there exists a constant $\mathbf{c}%
_{\Lambda}\geq1$ such that, if $\widehat{\gamma}_{A}(\cdot)=\widehat{\Gamma
}_{A}(\cdot;0)$ is the global heat kernel of
\[
\widehat{\mathcal{H}}= \textstyle\sum_{i,j=1}^{m}a_{i,j}\widehat{X}%
_{i}\widehat{X}_{j}-\partial_{t}
\]
(see Theorem \ref{thm.existenceGammaprop}), then
\begin{equation}
\label{eq.estiThmBLU}%
\begin{split}
&  \frac{1}{\mathbf{c}_{\Lambda}}\,t^{-Q/2}\,\exp\left(  -\mathbf{c}_{\Lambda}
\frac{\Vert z\Vert^{2}}{t}\right)  \leq\widehat{\gamma}_{A}(t,z)
\leq\mathbf{c}_{\Lambda}\,t^{-Q/2}\,\exp\left(  -\frac{\Vert z\Vert^{2}%
}{\mathbf{c}_{\Lambda}t}\right)  ,
\end{split}
\end{equation}
for every $t>0$, $z\in\mathbb{R}^{N}$ and $A\in\mathcal{M}_{\Lambda}$. Here,
$Q$ is the homogeneous dimension of $\mathbb{G}$ defined in \eqref{eq.homogQG}
and
\[
\Vert\cdot\Vert=d_{\widehat{X}}(\cdot;0),
\]
where $d_{\widehat{X}}$ is the CC distance associated with $\widehat{X}%
=\{\widehat{X}_{1},\ldots,\widehat{X}_{m}\}$. By com\-bi\-ning
\eqref{eq.estiThmBLU} with the representation formula
\eqref{eq.reprGammaASymm}, we then get
\begin{equation}
\label{eq.integralEstimGammaA}%
\begin{split}
&  \frac{1}{\mathbf{c}_{\Lambda}}\,(t-s)^{-Q/2}\int_{\mathbb{R}^{p}}
\exp\left(  -\mathbf{c}_{\Lambda}\frac{\Vert(x,0)^{-1} \ast(y,\xi)\Vert^{2}%
}{t-s}\right)  \,d\xi\leq\Gamma_{A}(t,x;s,y)\\
&  \qquad\leq\mathbf{c}_{\Lambda}\,\,(t-s)^{-Q/2}\int_{\mathbb{R}^{p}}
\exp\left(  -\frac{\Vert(x,0)^{-1}\ast(y,\xi)\Vert^{2}}{\mathbf{c}_{\Lambda
}(t-s)}\right)  d\xi,
\end{split}
\end{equation}
for every $x,y\in\mathbb{R}^{n}$, $t,s\in\mathbb{R}$ with $s<t$ and
$A\in\mathcal{M}_{\Lambda}$. We now exploit the results in \cite[Prop.s 4.2
and 4.4]{BiBraGauss}: there exists a constant $c_{0}>0$ such that, for any
$x,y\in\mathbb{R}^{n}$ and $t>0$,
\begin{equation}
\label{eq.estimBiBraInt}%
\begin{split}
&  \frac{1}{c_{0}\,|B_{X}(x,\sqrt{t})|}\,\exp\left(  -c_{0}\,\frac
{d_{X}(x,y)^{2}}{t}\right) \\[0.07cm]
&  \qquad\quad\leq t^{-Q/2}\int_{\mathbb{R}^{p}}\exp\left(  -\frac
{\Vert(x,0)^{-1}\ast(y,\xi)\Vert^{2}}{t}\right)  d\xi\\[0.07cm]
&  \qquad\qquad\qquad\leq\frac{c_{0}}{|B_{X}(x,\sqrt{t})|}\,\exp\left(
-\frac{d_{X}(x,y)^{2}}{c_{0}t}\right)  .
\end{split}
\end{equation}
Putting together \eqref{eq.integralEstimGammaA}, \eqref{eq.estimBiBraInt} and
the global doubling property of $d_{X}$ in \eqref{eq.globalD}, we immediately
obtain \eqref{eq.uniformGammaPar}. \medskip

(2)\thinspace\thinspace We distinguish three different cases. \medskip

\textsc{Case I:} $W_{1}\cdots W_{r}=X_{i_{1}}^{x}\cdots X_{i_{r}}^{x}$. In
this case, taking into account the re\-pre\-sen\-ta\-tion formula
\eqref{eq.derGammatx}, for every $x,y\in\mathbb{R}^{n}$ and $s<t$ we can
write
\begin{equation}
\label{eq.reprWiGammaCaseI}%
\begin{split}
&  \left\vert (\partial_{t})^{\alpha}(\partial_{s})^{\beta}\,W_{1}\cdots
W_{r}\Gamma_{A}(t,x;s,y)\right\vert \\
&  \qquad\leq\int_{\mathbb{R}^{p}}\left\vert (\partial_{t})^{\alpha+\beta}
\widehat{X}_{i_{1}}\cdots\widehat{X}_{i_{r}}\widehat{\gamma}_{A}\right\vert
\left(  t-s,(y,0)^{-1}\ast(x,\xi)\right)  d\xi.
\end{split}
\end{equation}
Moreover, reminding that the $\widehat{X}_{i}$'s are Lie-ge\-ne\-ra\-tors for
$\mathrm{Lie}(\mathbb{G})$, we can invoke \cite[Thm.\,2.5]{BLUpaper}: there
exists a constant $\mathbf{c}=\mathbf{c}_{\Lambda,r,\alpha,\beta}>0$ such
that, for every $z\in\mathbb{R}^{N}$, $t>0$ and $A\in\mathcal{M}_{\Lambda}$,
one has
\begin{equation}
\label{eq.estimDertzGI}\left\vert (\partial_{t})^{\alpha+\beta}\,\widehat{X}%
_{i_{1}}\cdots\widehat{X}_{i_{r}}\widehat{\gamma}_{A}(t,z)\right\vert
\leq\mathbf{c}\, t^{-(Q/2+\alpha+\beta+r/2)}\,\exp\left(  -\frac{\Vert
z\Vert^{2}}{\mathbf{c}\,t}\right)  .
\end{equation}
By combining \eqref{eq.reprWiGammaCaseI} with \eqref{eq.estimDertzGI} we then
get, for every $x,y\in\mathbb{R}^{n}$ and $s<t$,
\[%
\begin{split}
&  \left\vert (\partial_{t})^{\alpha}(\partial_{s})^{\beta}\,W_{1}\cdots
W_{r}\Gamma_{A}(t,x;s,y)\right\vert \\
&  \qquad\leq\frac{\mathbf{c}}{(t-s)^{\alpha+\beta+r/2}}\cdot(t-s)^{-Q/2}
\int_{\mathbb{R}^{p}}\exp\left(  -\frac{\Vert(y,0)^{-1}\ast(x,\xi)\Vert^{2}}
{\mathbf{c}(t-s)}\right)  d\xi.
\end{split}
\]
From this, by exploiting \eqref{eq.estimBiBraInt} and \eqref{eq.globalD} we
obtain
\begin{equation}
\label{eq.estimquasiFinalI}%
\begin{split}
&  \left\vert (\partial_{t})^{\alpha}(\partial_{s})^{\beta}\,W_{1}\cdots
W_{r}\Gamma_{A}(t,x;s,y)\right\vert \\
&  \qquad\leq{\kappa}\,\frac{(t-s)^{-(\alpha+\beta+r/2)}} {\left\vert
B_{X}(y,\sqrt{t-s})\right\vert }\, \exp\left(  -\frac{d_{X}(x,y)^{2}}%
{\kappa(t-s)}\right)  ,
\end{split}
\end{equation}
where $\kappa>0$ is a suitable constant only depending on $\Lambda,r$ and
$\alpha$. The desired \eqref{eq.uniformDerGammaPar} now follows from
\eqref{eq.estimquasiFinalI} and \eqref{eq.symmHxy}. \medskip

\textsc{Case II:} $W_{1}\cdots W_{r}=X_{j_{1}}^{y}\cdots X_{j_{r}}^{y}$. We
argue exactly as in Case I: by combining the integral representation formula
\eqref{eq.derGammays} with \eqref{eq.estimDertzGI}, we have
\begin{align*}
&  \left\vert (\partial_{t})^{\alpha}(\partial_{s})^{\beta}W_{1}\cdots
W_{r}\Gamma_{A}(t,x;s,y)\right\vert \\
&  \qquad\leq\int_{\mathbb{R}^{p}}\big\vert (\partial_{t})^{\alpha+\beta}
\,\widehat{X}_{j_{1}}\cdots\widehat{X}_{j_{r}}\widehat{\gamma}_{A}%
\big\vert \left(  t-s,(x,0)^{-1}\ast(y,\xi)\right)  d\xi\\[0.1cm]
&  \qquad\leq\frac{\mathbf{c}_{\Lambda,r,\alpha,\beta}}{(t-s)^{\alpha
+\beta+r/2}} \cdot(t-s)^{-Q/2}\int_{\mathbb{R}^{p}} \exp\left(  -\frac
{\Vert(x,0)^{-1}\ast(y,\xi)\Vert^{2}} {\mathbf{c}(t-s)}\right)  d\xi,
\end{align*}
for every $x,y\in\mathbb{R}^{n}$, $s<t$ and $A\in\mathcal{M}_{\Lambda}$. From
this, by taking into account \eqref{eq.estimBiBraInt} and \eqref{eq.globalD}
we immediately obtain \eqref{eq.uniformDerGammaPar}. \medskip

\textsc{Case III:} $W_{1}\cdots W_{r}= X_{j_{1}}^{y}\cdots X_{j_{k}}%
^{y}\,X_{i_{1}}^{x}\cdots X_{i_{h}}^{x}$ (with $h+k=r$). In this last case,
starting from the representation formula \eqref{eq.derGammatutte} we can
write
\begin{equation}
\label{eq.represmixedderCaseIII}%
\begin{split}
&  \left\vert (\partial_{t})^{\alpha}\,(\partial_{s})^{\beta}\,W_{1}\cdots
W_{r}\Gamma_{A}(t,x;s,y)\right\vert \\
&  \leq\int_{\mathbb{R}^{p}}\big|(\partial_{t})^{\alpha+\beta} \widehat{X}%
_{j_{1}}\cdots\widehat{X}_{j_{k}}\big( (\widehat{X}_{i_{1}}\cdots
\widehat{X}_{i_{h}}\widehat{\gamma}_{A})\circ\widetilde{\iota}%
\big) \big| \left(  t-s,(x,0)^{-1}\ast(y,\xi)\right)  d\xi.
\end{split}
\end{equation}
We should now apply some uniform e\-sti\-ma\-tes for the derivatives of
$\widehat{\gamma}_{A}$ as in \eqref{eq.estimDertzGI}. However, due to the
presence of the map $\widetilde{\iota}$ in \eqref{eq.represmixedderCaseIII},
such estimates are not directly available in \cite{BLUpaper}. We then exploit
Lemma \ref{lem.technicVF} to overcome this problem. \vspace*{0.08cm}

To begin with, for a fixed $t>0$ we consider the (smooth) function
\[
u_{t}:\mathbb{R}^{N}\rightarrow\mathbb{R},\qquad u_{t}(z):=\widehat{X}_{i_{1}}
\cdots\widehat{X}_{i_{h}}\widehat{\gamma}_{A}(t,z),
\]
and we repeatedly exploit \cite[Lem.\,5.3]{BiBraGauss}: this gives (see also
\eqref{eq.iotaiotatilde})
\begin{equation}
\label{eq.derKernelLemma}%
\begin{split}
&  \widehat{X}_{j_{1}}\cdots\widehat{X}_{j_{k}}\big( (\widehat{X}_{i_{1}}
\cdots\widehat{X}_{i_{h}}\widehat{\gamma}_{A}) \circ\widehat{\iota
}\big) =\widehat{X}_{j_{1}}\cdots\widehat{X}_{j_{k}}(u_{t}\circ\iota)\\
&  \qquad=(Z_{1}\cdots Z_{k}u_{t})\circ\iota=\big( Z_{1}\cdots Z_{k}\,
\widehat{X}_{i_{1}}\cdots\widehat{X}_{i_{h}}\,\widehat{\gamma}_{A}%
\big) \circ\widehat{\iota},
\end{split}
\end{equation}
where $Z_{1},\ldots,Z_{k}$ are suitable smooth vector fields $D_{\lambda}%
$-homogeneous of degree $1$ but not necessarily left-invariant. Since all the
vector fields in \eqref{eq.derKernelLemma} are $D_{\lambda}$-ho\-mo\-ge\-neous
of degree $1$, we are entitled to apply Lemma \ref{lem.technicVF}, obtaining
\begin{equation}
\label{eq.derKernelLemmaBis}Z_{1}\cdots Z_{k}\,\widehat{X}_{i_{1}}%
\cdots\widehat{X}_{i_{h}}\, \widehat{\gamma}_{A}=\!\!\!\sum_{r\,\leqslant
\,|\omega|\,\leqslant\, r\sigma_{n}}\!\!\!\!\gamma_{\omega}(z)\,\mathcal{Q}%
_{\alpha}(\widehat{X}_{1}, \ldots,\widehat{X}_{m})\widehat{\gamma}_{A}.
\end{equation}
Here, $r = h+k$, $\mathcal{Q}_{\omega}(\theta_{1},\ldots,\theta_{m})$ is a
homogeneous polynomial of (Euclidean) de\-gree $|\omega|$ in the non-commuting
variables $\theta_{1},\ldots,\theta_{m}$, and $\gamma_{\omega}$ is a
$D_{\lambda}$- homogeneous polynomial of degree $|\omega|-r$. On account of
\eqref{eq.derKernelLemmaBis}, we can then write
\begin{equation}
\label{eq.exprKernelQ}%
\begin{split}
&  (\partial_{t})^{\alpha+\beta}\widehat{X}_{j_{1}}\cdots\widehat{X}_{j_{k}}
\big( (\widehat{X}_{i_{1}}\cdots\widehat{X}_{i_{h}}\widehat{\gamma}_{A})
\circ\widetilde{\iota}\,\big)\\
&  =\sum_{r\,\leqslant\,|\omega|\,\leqslant\,r\sigma_{n}}(\partial
_{t})^{\alpha+\beta} \left[  \big(\gamma_{\omega}(\cdot)\mathcal{Q}_{\alpha}
(\widehat{X}_{1},\ldots,\widehat{X}_{m})\widehat{\gamma}_{A}\big) \circ
\widehat{\iota}\right] \\
&  =\sum_{r\,\leqslant\,|\omega|\,\leqslant\,r\sigma_{n}}\left[
\gamma_{\omega}(\cdot)\,\big( (\partial_{t})^{\alpha+\beta}\, \mathcal{Q}%
_{\omega}(\widehat{X}_{1},\ldots,\widehat{X}_{m}) \widehat{\gamma}%
_{A}\big)\right]  \circ\widehat{\iota}.
\end{split}
\end{equation}
We now observe that, by \eqref{eq.estimDertzGI}, for all $z\in\mathbb{R}^{N}$
and $t>0$ we have
\[
\big\vert (\partial_{t})^{\alpha+\beta}\, \mathcal{Q}_{\omega}(\widehat{X}%
_{1},\ldots,\widehat{X}_{m})\widehat{\gamma}_{A}(t,z)\big\vert \leq
\mathbf{c}\,t^{-Q/2-\alpha-\beta-|\omega|/2}\, \exp\left(  -\frac{\Vert
z\Vert^{2}}{\mathbf{c}\,t}\right)  ,
\]
where the constant $\mathbf{c}$ only depends on $\Lambda,r,\alpha$ and $\beta
$. Furthermore, since the fun\-ction $\gamma_{\omega}$ is smooth and
$D_{\lambda}$-homogeneous of degree $|\omega|-r$, a simple ho\-mo\-ge\-nei\-ty
argument shows that (see, e.g., \cite[Prop.\,5.1.4]{BLUlibro})
\begin{equation}
\label{eq.estimgammaalphadX}|\gamma_{\omega}(z)|\leq\mu\,\Vert z\Vert
^{|\omega|-r}\qquad(\text{for all $z\in\mathbb{R}^{N}$}),
\end{equation}
where $\mu>0$ is a `structural' constant which can be chosen independently of
$\omega$. Ga\-the\-ring \eqref{eq.exprKernelQ}-\eqref{eq.estimgammaalphadX},
and bearing in mind the very definition of $\widehat{\iota}$ in
\eqref{eq.iotaiotatilde}, we then obtain the estimate \vspace*{0.05cm}
\begin{equation}
\label{eq.crucialderrhoCaseIII}%
\begin{split}
&  \big\vert (\partial_{t})^{\alpha+\beta}\widehat{X}_{j_{1}}\cdots
\widehat{X}_{j_{k}}\big( (\widehat{X}_{i_{1}}\cdots\widehat{X}_{i_{h}}
\widehat{\gamma}_{A})\circ\widetilde{\iota}\big) \big\vert\\
&  \leq\varrho\,t^{-Q/2-r/2-\alpha-\beta}\!\!\!\!\!\sum_{r\,\leqslant
\,|\omega|\,\leqslant\,r\sigma_{n}}\left(  \frac{\Vert z^{-1}\Vert}{ \sqrt{t}%
}\right)  ^{{|\omega|-r}}\,\exp\left(  -\frac{\Vert z^{-1}\Vert^{2}}
{\mathbf{c}\,t}\right) \\
&  \leq\varrho_{1}\,t^{-Q/2-r/2-\alpha-\beta}\,\exp\left(  -\frac{\Vert
z^{-1}\Vert^{2}} {\varrho_{1}t}\right) \\[0.1cm]
&  (\text{by Remark \ref{rem.dXhomog}})\\[0.1cm]
&  =\varrho_{1}\,t^{-Q/2-r/2-\alpha-\beta}\,\exp\left(  -\frac{\Vert
z\Vert^{2}} {\varrho_{1}t}\right)  ,
\end{split}
\end{equation}
holding true for every $z\in\mathbb{R}^{N}$, $t>0$ and $A\in\mathcal{M}%
_{\Lambda}$ (here, $\rho_{1}>0$ is constant only depending on $\Lambda
,r,\alpha$ and $\beta$). Finally, by combining \eqref{eq.crucialderrhoCaseIII}
with \eqref{eq.represmixedderCaseIII}, we get
\begin{align*}
&  \left\vert (\partial_{t})^{\alpha}\,(\partial_{s})^{\beta}\, W_{1}\cdots
W_{r}\Gamma_{A}(t,x;s,y)\right\vert \\
&  \leq\frac{\rho_{1}}{(t-s)^{\alpha+\beta+r/2}}\cdot(t-s)^{-Q/2}
\int_{\mathbb{R}^{p}}\exp\left(  -\frac{\Vert(x,0)^{-1}\ast(y,\xi) \Vert^{2}%
}{\varrho_{1}\,(t-s)}\right)  d\xi,
\end{align*}
for every $x,y\in\mathbb{R}^{n}$, $s<t$ and $A\in\mathcal{M}_{\Lambda}$. From
this, by taking into account \eqref{eq.estimBiBraInt} and \eqref{eq.globalD},
we obtain \eqref{eq.uniformDerGammaPar}. \medskip

(3)\,\,As for the proof of \eqref{eq.uniformDerGammaPar}, we distinguish three
cases. \medskip

\textsc{Case I:} $W_{1}\cdots W_{r}=X_{i_{1}}^{x}\cdots X_{i_{r}}^{x}$. In
this case, by using the integral re\-pre\-sen\-ta\-tion formula
\eqref{eq.derGammatx} for both $\Gamma_{A}$ and $\Gamma_{B}$, we have the
estimate
\begin{equation}
\label{eq.reprWiGammaABCaseI}%
\begin{split}
&  \left\vert (\partial_{t})^{\alpha}(\partial_{s})^{\beta}\,W_{1}\cdots
W_{r}(\Gamma_{A}-\Gamma_{B})(t,x;s,y)\right\vert \\[0.1cm]
&  \leq\int_{\mathbb{R}^{p}}\big\vert (\partial_{t})^{\alpha+\beta}
\,\widehat{X}_{i_{1}}\cdots\widehat{X}_{i_{r}}(\widehat{\gamma}_{A}
-\widehat{\gamma}_{B})\big\vert \left(  t-s,(y,0)^{-1}\ast(x,\xi)\right)
d\xi,
\end{split}
\end{equation}
for every $x,y\in\mathbb{R}^{n}$, $s<t$ and $A,B\in\mathcal{M}_{\Lambda}$. On
the other hand, since the $\widehat{X}_{i}$'s are Lie-generators for
$\mathrm{Lie}(\mathbb{G})$, we can apply once again \cite[Thm.\,2.5]%
{BLUpaper}: there exists a constant $\mathbf{c}=\mathbf{c}_{\Lambda
,r,\alpha,\beta}>0$ such that
\begin{equation}
\label{eq.MainEstimKernelAB}%
\begin{split}
&  \big\vert (\partial_{t})^{\alpha+\beta}\,\widehat{X}_{i_{1}} \cdots
\widehat{X}_{i_{r}}\widehat{\gamma}_{A}(t,z)-(\partial_{t})^{\alpha+\beta}
\,\widehat{X}_{i_{1}}\cdots\widehat{X}_{i_{r}}\widehat{\gamma}_{B}%
(t,z)\big\vert\\
&  \leq\mathbf{c}\,\| A-B\|^{1/\sigma_{n}}\,t^{-Q/2-\alpha-\beta-r/2}
\exp\left(  -\frac{\Vert z\Vert^{2}}{\mathbf{c}\,t}\right)  .
\end{split}
\end{equation}
With reference to \cite[Thm.\,2.5]{BLUpaper}, the exponent $1/\sigma_{n}$ is
justified by the fact that the step of nilpotency of $\mathbb{G}$ is precisely
$r=\sigma_{n}$ (see Theorem \ref{thm.liftingGroupG}). \vspace*{0.07cm}

Gathering \eqref{eq.reprWiGammaABCaseI} and \eqref{eq.MainEstimKernelAB}, we
then get
\begin{equation}
\label{eq.uniformGammaABCaseIQuasiFinal}%
\begin{split}
&  \left\vert (\partial_{t})^{\alpha}(\partial_{s})^{\beta}\,W_{1}\cdots W_{r}
(\Gamma_{A}-\Gamma_{B})(t,x;s,y)\right\vert \\
&  \quad\leq\mathbf{c}\left\Vert A-B\right\Vert ^{1/\sigma_{n}}(t-s)^{-\alpha
-\beta-r/2}\times\\
&  \quad\quad\times(t-s)^{-Q/2}\int_{\mathbb{R}^{p}} \exp\left(  -\frac
{\Vert(y,0)^{-1}\ast(x,\xi)\Vert^{2}}{\mathbf{c}\,(t-s)}\right)  \,d\xi\\
&  \quad\text{(by \eqref{eq.estimBiBraInt} and \eqref{eq.globalD})}\\
&  \quad\leq\mathbf{c}\,\left\Vert A-B\right\Vert ^{1/\sigma_{n}}\,
\frac{(t-s)^{-\alpha-\beta-r/2}}{\left\vert B_{X}(y,\sqrt{t-s})\right\vert }
\,\exp\left(  -\frac{d_{X}(x,y)^{2}}{\kappa(t-s)}\right)  ,
\end{split}
\end{equation}
for every $x,y\in\mathbb{R}^{n}$, $s<t$ and $A,B\in\mathcal{M}_{\Lambda}$.
Here $\kappa>0$ is constant only depending on $\Lambda,r,\alpha$ and $\beta$.
The desired \eqref{eq.uniformGammaABPar} is now a consequence of
\eqref{eq.uniformGammaABCaseIQuasiFinal} and \eqref{eq.symmHxy}. \medskip

\textsc{Case II:} $W_{1}\cdots W_{r}=X_{j_{1}}^{y}\cdots X_{j_{r}}^{y}$. This
is very similar to Case I (see also Case II in the proof of (2)), and we omit
the details. \medskip

\textsc{Case III:} $W_{1}\cdots W_{r}= X_{j_{1}}^{y}\cdots X_{j_{k}}%
^{y}\,X_{i_{1}}^{x}\cdots X_{i_{h}}^{x}$ (with $h+k=r$). In this last case, by
com\-bi\-ning \eqref{eq.exprKernelQ} with the representation formula
\eqref{eq.derGammatutte} for $\Gamma_{A}$ and $\Gamma_{B}$, we can write, for
every $x,y\in\mathbb{R}^{n}$, $s<t$ and $A,B\in\mathcal{M}_{\Lambda}$,
\[%
\begin{split}
&  \big\vert (\partial_{t})^{\alpha}(\partial_{s})^{\beta}\, W_{1}\cdots
W_{r}(\Gamma_{A}-\Gamma_{B})(t,x;s,y)\big\vert\\
&  \quad\leq\int_{\mathbb{R}^{p}}\big\vert (\partial_{t})^{\alpha+\beta}
\widehat{X}_{j_{1}}\cdots\widehat{X}_{j_{k}}\big( (\widehat{X}_{i_{1}}
\cdots\widehat{X}_{i_{h}}\widehat{\gamma}_{A})\circ\widetilde{\iota}\big)\\
&  \qquad\qquad-(\partial_{t})^{\alpha+\beta}\widehat{X}_{j_{1}}%
\cdots\widehat{X}_{j_{k}}\big( (\widehat{X}_{i_{1}}\cdots\widehat{X}_{i_{h}}
\widehat{\gamma}_{B})\circ\widetilde{\iota}\big)\big\vert d\xi\,\\
&  \quad= \sum_{r\,\leqslant\,|\omega|\,\leqslant\,r\sigma_{n}} \int%
_{\mathbb{R}^{p}} \Big\vert \big[ \gamma_{\omega}(\cdot)\big((\partial
_{t})^{\alpha+\beta} \,\mathcal{Q}_{\omega}(\widehat{X}_{1},\ldots
,\widehat{X}_{m}) (\widehat{\gamma}_{A}-\widehat{\gamma}_{B})\big)\big] \circ
\widehat{\iota}\, \Big\vert d\xi,
\end{split}
\]
where all the integrand functions are evaluated at $\left(  t-s,(x,0)^{-1}%
\ast(y,\xi)\right)  . $ On the other hand, by applying
\eqref{eq.MainEstimKernelAB} to each monomial in $\mathcal{Q}_{\omega}$, we
have
\begin{equation}
\label{eq.estiQalphagammaAB}%
\begin{split}
&  \big\vert (\partial_{t})^{\alpha+\beta}\,\mathcal{Q}_{\omega}
(\widehat{X}_{1},\ldots,\widehat{X}_{m})\widehat{\gamma}_{A}(t,z)-
(\partial_{t})^{\alpha+\beta}\,\mathcal{Q}_{\omega}(\widehat{X}_{1},
\ldots,\widehat{X}_{m})\widehat{\gamma}_{B}(t,z)\big\vert\\
&  \qquad\leq\mathbf{c}\,\| A-B\|^{1/\sigma_{n}}\,t^{-Q/2-\alpha-\beta
-|\omega|/2} \exp\left(  -\frac{\Vert z\Vert^{2}}{\mathbf{c}\,t}\right)  ,
\end{split}
\end{equation}
where the constant $\mathbf{c}=\mathbf{c}_{\Lambda,r,\alpha,\beta}>0$ can be
chosen independently of $\omega$. Gathering \eqref{eq.estiQalphagammaAB},
\eqref{eq.estimgammaalphadX} and the definition of $\widehat{\iota}$, we then
obtain
\begin{equation}
\label{eq.estimKergammaABQFinal}%
\begin{split}
&  \sum_{r\,\leqslant\,|\omega|\,\leqslant\,r\sigma_{n}}
\Big\vert \big[ \gamma_{\omega}(\cdot)\big((\partial_{t})^{\alpha+\beta}
\,\mathcal{Q}_{\omega}(\widehat{X}_{1},\ldots,\widehat{X}_{m})
(\widehat{\gamma}_{A}-\widehat{\gamma}_{B})\big)\big] \circ\widehat{\iota}\,
\Big\vert(t,z)\\
&  \qquad\leq\varrho\,\| A-B\|^{1/\sigma_{n}}\,t^{-Q/2-\alpha-\beta-r/2}
\times\\
&  \qquad\qquad\times\sum_{r\,\leqslant\,|\omega|\,\leqslant\,r\sigma_{n}}
\left(  \frac{\Vert z^{-1}\Vert}{\sqrt{t}}\right)  ^{|\omega|-r} \exp\left(
-\frac{\Vert z^{-1}\Vert^{2}}{\mathbf{c}\,t}\right) \\
&  \qquad\leq\varrho_{1}\,\| A-B\|^{1/\sigma_{n}}\,t^{-Q/2-\alpha-\beta-r/2}
\exp\left(  -\frac{\Vert z^{-1}\Vert^{2}}{\varrho_{1}t}\right) \\
&  \qquad(\text{by Remark \ref{rem.dXhomog}})\\[0.1cm]
&  \qquad= \varrho_{1}\,\| A-B\|^{1/\sigma_{n}}\,t^{-Q/2-\alpha-\beta-r/2}
\exp\left(  -\frac{\Vert z\Vert^{2}}{\varrho_{1}t}\right)  ,
\end{split}
\end{equation}
for every $z\in\mathbb{R}^{N}$, $t>0$ and $A,B\in\mathcal{M}_{\Lambda}$. Here
$\rho_{1}>0$ is a constant only depending on $\Lambda,r,\alpha$ and $\beta$.
With \eqref{eq.estimKergammaABQFinal} at hand, we can establish
\eqref{eq.uniformGammaABPar} as in {Case III} of (2). This completes the proof.
\end{proof}

\section{Operators with H\"older-continuous coefficients}

\label{sec.Holdervar}The aim of this section is to prove existence and several
`structural properties' of a {global heat kernel} for the variable coefficient
operator (\ref{H var}). Our proof of Theorem \ref{thm.thmVariableCoeff} is
based on a suitable adaptation of the celebrated method developed by E.E. Levi
to study uniformly elliptic equations of order $2n$ (see \cite{Levi1, Levi2}),
and later extended to the uniformly parabolic equations (see \cite{Fried}). As
explained in the Introduction, this approach has been already exploited in
\cite{BBLUMemoir} to prove an analog of Theorem \ref{thm.thmVariableCoeff} for
\emph{generic parabolic H\"{o}rmander operators}
\[
H=\sum_{i,j=1}^{m}\alpha_{i,j}(t,x)X_{i}X_{j}+\sum_{k=1}^{m}\alpha
_{k}(t,x)X_{k}+\alpha_{0}(t,x)-\partial_{t},
\]
under the following `structural assumptions':

\begin{itemize}
\item[(a)] $X_{1},\ldots,X_{m}$ are smooth vector fields on $\mathbb{R}^{n}$
(for some $m=k+n$) and they satisfy H\"{o}r\-man\-der's rank condition at
every point of $\mathbb{R}^{n}$;

\item[(b)] the coefficient functions of $H$ are globally
H\"older-con\-ti\-nuo\-us, and
\begin{equation}
\label{eq.assumptMemoiraij}\alpha_{i,j}(t,x) \equiv%
\begin{cases}
1, & \text{if $i = j$},\\
0, & \text{if $i\neq j$},
\end{cases}
\end{equation}
for every $k+1\leq i,j\leq k+n$;

\item[(c)] there exists a bounded domain $\Omega_{0}\subseteq\mathbb{R}^{n}$
such that
\begin{equation}
\label{eq.assumptionMemoir}\text{$\left(  X_{1},\ldots,X_{k},X_{k+1}%
,\ldots,X_{k+n}\right)  \equiv\left(  0,\ldots,0,\partial_{x_{1}}%
,\ldots,\partial_{x_{n}}\right)  $ \,\,on $\mathbb{R}^{n}\setminus\Omega_{0}$%
}.
\end{equation}

\end{itemize}

Clearly, the hypotheses of our Theorem \ref{thm.thmVariableCoeff} do not
necessarily imply \eqref{eq.assumptMemoiraij}-\eqref{eq.assumptionMemoir};
how\-ever, these `structural assumptions' play a (key) r\^{o}le only in
\cite[Part I]{BBLUMemoir}, where the Authors carry out a careful analysis of
the constant co\-ef\-fi\-cient operator cor\-re\-spon\-ding to $H$ in order to
e\-sta\-blish the analog of Theorems \ref{thm.existenceGammaprop}%
-\ref{thm.mainParabolicCase}.

Since in our homogeneous setting we have been able to study constant
coefficient operators without requiring \eqref{eq.assumptionMemoir} (and with
a totally different approach), we can prove Theorem \ref{thm.thmVariableCoeff}
by proceeding \emph{verbatim} as in \cite[Part II]{BBLUMemoir}: what we only
need to check is that all the `structural' ingredients used in
\cite{BBLUMemoir} to set up the Levi method are satisfied in our context. We
devote the rest of this section to this aim.

\subsection{Heat kernel for constant coefficients operators}

\label{subsec.UnifEstim} A first fundamental ingredient for the Levi method in
\cite{BBLUMemoir} is the existence of a `well-behaved' (global) heat kernel
for the constant coefficient operator obtained by freezing the
co\-ef\-fi\-cients $a_{i,j}$ (but not the $X_{i}$'s) at any point
$(t_{0},x_{0})\in\mathbb{R}^{1+n}$, that is,
\begin{equation}
\label{eq.OpFrozen}\mathcal{H}_{(t_{0},x_{0})}= \sum_{i,j=1}^{m}a_{i,j}%
(t_{0},x_{0})X_{i}X_{j}-\partial_{t},
\end{equation}
together with some \emph{uniform estimates} of this kernel with respect to
$(t_{0},x_{0})\in\mathbb{R}^{1+n}$ (see, precisely, \cite[Thm.\,10.10]%
{BBLUMemoir}). These results in our context are contained in Theorems
\ref{thm.existenceGammaprop} and \ref{thm.mainParabolicCase}.

Let us now introduce a convenient notation which shall play a key r\^{o}le in
the sequel: following \cite{BBLUMemoir}, we denote by $\mathbf{E}$ the
\emph{Gaussian-type function}
\begin{equation}
\label{eq.defdXGaussE}\mathbf{E}(x,y,t):=\frac{1}{|B_{X}(x,\sqrt{t})|}\,
\exp\left(  -\frac{d_{X}(x,y)^{2}}{t}\right)  \qquad(x,y\in\mathbb{R}%
^{n},\,t>0).
\end{equation}
We explicitly notice that, by exploiting Proposition
\ref{Prop global doubling}, for every $\kappa>0$ there exists $c=c(\kappa)>0$
such that
\[%
\begin{split}
&  c(\kappa)^{-1}\,\frac{1}{|B_{X}(x,\sqrt{t})|}\, \exp\left(  -\frac
{d_{X}(x,y)^{2}}{\kappa\,t}\right)  \leq\mathbf{E}(x,y,\kappa t)\\
&  \qquad\qquad\leq c(\kappa)\,\frac{1}{|B_{X}(x,\sqrt{t})|}\, \exp\left(
-\frac{d_{X}(x,y)^{2}}{\kappa\,t}\right)  ,
\end{split}
\]
for any $x,y\in\mathbb{R}^{n}$ and any $t>0$. With this notation at hand, we
can rewrite the Gaussian bounds of Theorem \ref{thm.mainParabolicCase} as follows.

\begin{corollary}
\label{cor.estimconE} Under the assumptions of Theorem
\ref{thm.thmVariableCoeff}, for every $(t_{0},x_{0})\in\mathbb{R}^{1+n}$ let
us denote by $\Gamma_{(t_{0},x_{0})}$ the global heat kernel of the constant
coefficient operator \eqref{eq.OpFrozen}. Then, the following estimates hold.

\begin{enumerate}
\item There exist constants $\kappa_{\Lambda},\,\nu_{\Lambda}>0$ such that
\begin{equation}
\label{eq.uniformGaussianFrozenEE}\frac{1}{\nu_{\Lambda}}\,\mathbf{E}%
(x,y,\kappa_{\Lambda}^{-1}(t-s))\leq\Gamma_{(t_{0},x_{0})}(t,x;s,y)\leq
\nu_{\Lambda}\, \mathbf{E}(x,y,\kappa_{\Lambda}(t-s)),
\end{equation}
for every $x,y\in\mathbb{R}^{n}$, every $s<t$ and every $(t_{0},x_{0})
\in\mathbb{R}^{1+n}$.

\item For every integer $r\geq1$, $i_{1},\ldots,i_{r}\in\{1,\ldots,m\}$ and
every integer $\alpha\geq0$, there exist constants $\kappa=\kappa
_{\Lambda,r,\alpha},\,\nu=\nu_{\Lambda,r,\alpha}>0$ such that
\[
\left\vert (\partial_{t})^{\alpha}\,X_{i_{1}}^{x}\cdots X_{i_{r}}^{x}
\Gamma_{(t_{0},x_{0})}(t,x;s,y)\right\vert \leq\nu\,(t-s)^{-(\alpha+r/2)}\,
\mathbf{E}(x,y,\kappa(t-s)),
\]
\begin{equation}
\label{eq.uniformABFrozenEE}%
\begin{split}
&  \left\vert (\partial_{t})^{\alpha}\,X_{i_{1}}^{x}\cdots X_{i_{r}}^{x}
\left(  \Gamma_{(t_{0},x_{0})}(t,x;s,y)-\Gamma_{(t_{1},x_{1})}
(t,x;s,y)\right)  \right\vert \\[0.15cm]
&  \quad\leq{\nu}\, d_{P}\left(  (t_{0},x_{0}),(t_{1},x_{1})\right)
^{\alpha/\sigma_{n}}(t-s)^{-(\alpha+r/2)} \,\mathbf{E}(x,y,{\kappa}(t-s)),
\end{split}
\end{equation}
for every $x,y\in\mathbb{R}^{n}$, $s<t$ and $(t_{0},x_{0}),(t_{1},x_{1}%
)\in\mathbb{R}^{1+n}$.
\end{enumerate}
\end{corollary}

As already pointed out, all the assertions in the above Corollary are actually
contained in Theorems \ref{thm.existenceGammaprop} and
\ref{thm.mainParabolicCase}, since our assumptions imply that
\[
A(t_{0},x_{0})=
\begin{pmatrix}
a_{i,j}(t_{0},x_{0})
\end{pmatrix}
_{i,j=1}^{m}\in\mathcal{M}_{\Lambda}\qquad\text{for every $(t_{0},x_{0}%
)\in\mathbb{R}^{1+n}$}.
\]
As for estimate \eqref{eq.uniformABFrozenEE}, if follows from
\eqref{eq.uniformGammaABPar} and the fact that, since the $a_{i,j}$'s are
globally H\"{o}lder-continuous (see assumption (i) in Theorem
\ref{thm.thmVariableCoeff}), we have
\begin{align*}
\| A(t_{0},x_{0})-A(t_{1},x_{1})\|  &  \leq\left(  \sum_{i,j=1}^{m}%
|a_{i,j}(t_{0},x_{0})-a_{i,j}(t_{1},x_{1})|^{2}\right)  ^{1/2}\\[0.1cm]
&  \leq K\,d_{P}\left(  (t_{0},x_{0}),(t_{1},x_{1})\right)  ^{\alpha},
\end{align*}
where we have set $K:=\max_{i,j}\Vert a_{i,j}\Vert_{\alpha,\,\mathbb{R}^{1+n}%
}$. \medskip

\subsection{Metric properties of $d_{X}$}

\label{subsec.metricPropdX} A second important ingredient for the argument in
\cite{BBLUMemoir} is the validity of the following estimates for the
CC-distance $d_{X}$, which in \cite{BBLUMemoir} are heavily based on
assumption \eqref{eq.assumptionMemoir} (see \cite[Lem.\,2.4]{BBLUMemoir}):

\begin{itemize}
\item there exists a constant $\mathbf{c}>0$ such that
\begin{equation}
\label{eq.dXgreaterdE}d_{X}(x,y)\geq\mathbf{c}|x-y|\qquad\text{$\forall
\,\,x,y\in\mathbb{R}^{n} $};
\end{equation}

\item for every $\sigma>0$ there exists a constant $\mathbf{c}(\sigma)>0$ such
that
\begin{equation}
d_{X}(x,y)\leq\mathbf{c}(\sigma)|x-y| \quad\forall\,\,x,y\in\mathbb{R}%
^{n}:\,\max\{|x-y|,d_{X}(x,y)\}\geq\sigma. \label{eq.dXlowerdEloc}%
\end{equation}

\end{itemize}

Starting from \eqref{eq.dXgreaterdE}-\eqref{eq.dXlowerdEloc}, it is possible
to prove several global properties of the `geometry' of $(\mathbb{R}^{n}%
,d_{X})$ which play a key r\^{o}le in the analysis of $H$. \vspace*{0.05cm}

On the other hand, by carefully scrutinizing the proofs in \cite[Part
II]{BBLUMemoir}, it is easy to recognize that one \emph{does not really need}
\eqref{eq.dXgreaterdE}-\eqref{eq.dXlowerdEloc}: in fact, the only properties
of $d_{X}$ which intervene in the Levi method are the following:

\begin{enumerate}
\item[(a)] any $B_{X}$-ball is bounded in the Euclidean sense;

\item[(b)] there exist constants $\mathbf{c}>0$ and $\vartheta\geq n$ such
that
\[
|B_{X}(x,Mr)|\leq\mathbf{c}M^{\vartheta}|B_{X}(x,r)|\qquad\forall\,\,
M\geq1,\,x\in\mathbb{R}^{n},\,r>0;
\]

\item[(c)] for every $R>0$ there exists a constant $\mathbf{c}(R)>0$ such
that
\[
|B_{X}(x,r)|\geq\mathbf{c}(R)\,r^{\vartheta}\qquad\forall\,\,0<r\leq R,
\,x\in\mathbb{R}^{n},
\]
where $\vartheta$ is as in (b);

\item[(d)] for every $\sigma>0$ there exists a constant $\mathbf{c}(\sigma)>0$
such that
\[
|B_{X}(x,r)|\geq\mathbf{c}(\sigma)\,r^{n}\qquad\forall\,\,r\geq\sigma
>0,\,x\in\mathbb{R}^{n}.
\]

\end{enumerate}

Even if we cannot expect that \eqref{eq.dXgreaterdE}-\eqref{eq.dXlowerdEloc}
hold in our homogeneous setting, the next proposition shows that properties
(a)-(d) are still satisfied by $d_{X}$.

\begin{proposition}
\label{thm.propdXnostre} Let the assumptions and the notation of Theorem
\ref{thm.thmVariableCoeff} do apply. In particular, let $q=\sum_{i=1}%
^{n}\sigma_{i}\geq n$ be the $\delta_{\lambda}$-homogeneous dimension of
$\mathbb{R}^{n}$. \vspace*{0.05cm}

Then, the following facts hold.

\begin{enumerate}
\item A subset $B\subseteq\mathbb{R}^{n}$ is bounded with respect to $d_{X}$
\emph{if and only if} it is bounded with respect to the Euclidean distance.

\item There exists a constant $\gamma>0$ such that
\begin{equation}
\label{eq.doublingdXneededMem}|B_{X}(x,Mr)|\leq\gamma\,M^{q}|B(x,r)|\qquad
\forall\,\,M\geq1,\,x\in\mathbb{R}^{n},\,r>0.
\end{equation}

\item There exists a constant $\omega>0$ such that
\begin{equation}
\label{eq.lowerboundvolMem}|B_{X}(x,r)|\geq\omega\,r^{q}\qquad\forall
\,\,x\in\mathbb{R}^{n},\,r>0.
\end{equation}

\item For every $\sigma>0$ there exists a constant $\mathbf{c}=\mathbf{c}%
(\sigma)>0$ such that
\begin{equation}
\label{eq.lowerboundvolMemnn}|B_{X}(x,r)|\geq\mathbf{c}(\sigma)\,r^{n}
\qquad\forall\,\,x\in\mathbb{R}^{n},\,r\geq\sigma.
\end{equation}

\end{enumerate}
\end{proposition}

\begin{proof}
(1)\thinspace\thinspace We first suppose that $B\subseteq\mathbb{R}^{n}$ is
bounded in the Euclidean sense. Since the map $x\mapsto d_{X}(0,x)$ is
continuous in the Euclidean sense (see Remark \ref{rem.continuitysame}) and
$\overline{B}$ is compact, there exists $r>0$ such that
\[
d_{X}(0,x)\leq r\qquad\text{for every $x\in\overline{B}$},
\]
thus $B$ is bounded with respect to $d_{X}$. \vspace*{0.03cm}

Assume now that $B\subseteq\mathbb{R}^{n}$ is $d_{X}$-bounded, and let $R > 0$
be such that that
\[
B\subseteq B_{X}(0,R).
\]
Reminding that $X_{1},\ldots,X_{m}$ satisfy H\"{o}rmander's condition in
$\mathbb{R}^{n}$, we know that there exists a small $\rho>0$ such that
$B_{X}(0,\rho)$ is bounded in Euclidean sense (see, e.g., \cite[Prop.\,7.21]%
{BBBook}); on the other hand, since $X_{i}$'s are also $\delta_{\lambda}%
$-ho\-mo\-ge\-neous of degree $1$, from Remark \ref{rem.dXhomog}-(2) we infer
that
\[
B_{X}(0,R)=\delta_{R/\rho}\left(  B_{X}(0,\rho)\right)  .
\]
Since $\delta_{R/\rho}$ is a (linear) diffeomorphism of $\mathbb{R}^{n}$ we
deduce that $B_{X}(0,R)$ is boun\-ded in the Euclidean sense, and thus the
same is true of $B$. \medskip

(2)\thinspace\thinspace Inequality \eqref{eq.doublingdXneededMem} immediately
follows from \eqref{eq.globalD}. \medskip

(3)\thinspace\thinspace First of all, since $X_{1},\ldots,X_{m}$ satisfy
assumptions (H.1) and (H.2), the fol\-low\-ing \emph{global version} of a
celebrated result by Nagel, Stein and Wainger \cite{NSW} holds true (see
\cite[Thm.\,B]{BiBoBra1}): there exist $c_{1},c_{2}>0$ such that
\begin{equation}
\label{eq.globalNSWBX}c_{1}\sum_{k=n}^{q}f_{k}(x)r^{k}\leq|B_{X}(x,r)|\leq
c_{2}\sum_{k=n}^{q} f_{k}(x)r^{k}\qquad\forall\,\,x\in\mathbb{R}^{n},\,r>0,
\end{equation}
where, for any $k\in\{n,\ldots,q\}$, the function $f_{k}:\mathbb{R}%
^{n}\rightarrow\mathbb{R}$ is continuous, non-ne\-ga\-ti\-ve and
$\delta_{\lambda}$-homogeneous of degree $q-k$; in particular, $f_{q}(x)$ is
constant in $x$ and strictly positive. As a consequence, setting
$\omega:=f_{q}>0$, from \eqref{eq.globalNSWBX} we get
\eqref{eq.lowerboundvolMem}. \medskip

(4)\thinspace\thinspace By making use of \eqref{eq.lowerboundvolMem}, and
reminding that $q\geq n$, we have
\[
|B_{X}(x,r)|\geq\omega\,r^{q}=\omega\,r^{n}\,r^{q-n}\geq\omega\,\sigma
^{q-n}\,r^{n},
\]
for every $x\in\mathbb{R}^{n}$ and $r\geq\sigma>0$. This gives
\eqref{eq.lowerboundvolMemnn}, and the proof is complete.
\end{proof}

Thought not explicit stated, there is another key property concerning $d_{X}$
which is repeatedly exploited in \cite{BBLUMemoir}: for every fixed
\emph{$x\in\mathbb{R}^{n}$} it holds that
\begin{equation}
\label{eq.summexpdX}y\mapsto e^{-d_{X}(x,y)^{2}}\in L^{1}(\mathbb{R}^{n}).
\end{equation}
While in \cite{BBLUMemoir} this is an immediate consequence of
\eqref{eq.dXgreaterdE} (which, in its turn, follows from assumption
\eqref{eq.assumptionMemoir}), in our context we need to prove
\eqref{eq.summexpdX} directly.

\begin{lemma}
\label{lem.sumexpdXhom} For every $p\geq1$ and every fixed $x\in\mathbb{R}%
^{n}$, we have
\[
\mathbf{e}_{x}(y):=e^{-d_{X}(x,y)^{2}}\in L^{p}(\mathbb{R}^{n}).
\]

\end{lemma}

\begin{proof}
We first observe, since $0<\mathbf{e}_{x}\leq1$ on $\mathbb{R}^{n}$, we
obviously have $\mathbf{e}_{x}\in L_{\operatorname{loc}}^{p}(\mathbb{R}^{n})$.
Thus, reminding that any $d_{X}$-ball is bounded in the Euclidean sense (as we
know from Theo\-rem \ref{thm.propdXnostre}), to prove the lemma it suffices to
demonstrate that
\begin{equation}
\label{eq.summexpdinfinity}\mathbf{e}_{x}\in L^{p}(\mathbb{R}^{n}\setminus
B_{X}(0,1)).
\end{equation}
To establish \eqref{eq.summexpdinfinity} we notice that, by triangle's
inequality for $d_{X}$, we have
\[
d_{X}(0,y)^{2}\leq2d_{X}(x,y)^{2}+2d_{X}(0,x)^{2},
\]
and thus
\[
\mathbf{e}_{x}(y)\leq e^{d_{X}(0,x)^{2}} \cdot e^{-\frac{1}{2}d_{X}(0,y)^{2}}
= c_{x}\,e^{-\frac{1}{2}\,d_{X}(0,y)^{2}}\quad\forall\,\,y\in\mathbb{R}^{n}.
\]
As a consequence, since the function $y\mapsto d_{X}(0,y)$ is $\delta
_{\lambda}$-homogeneous of degree $1$ (see Remark \ref{rem.dXhomog}), we
obtain the following computation:
\begin{align*}
\int_{\mathbb{R}^{n}\setminus B_{X}(0,1)}\mathbf{e}_{x}^{p}\,\left(  y\right)
dy  &  \leq c_{x}^{p}\,\int_{\{y:\,d_{X}(0,y)\geq1\}}e^{-\frac{p}{2}%
\,d_{X}(0,y)^{2}}\,dy\\
&  =c_{x}^{p}\,\sum_{k=0}^{\infty}\int_{\{y:\,2^{k}\leq d_{X}(0,y)<2^{k+1}\}}
e^{-\frac{p}{2}\,d_{X}(0,y)^{2}}\,dy\\
&  (\text{using the change of variable $y=\delta_{2^{k}}(u)$})\\
&  =c_{x}^{p}\sum_{k=0}^{\infty} \int_{\{u:\,1\leq d_{X}(0,u)<2\}}%
e^{-2^{2k-1}\,p\,d_{X}(0,u)^{2}}\cdot2^{kq}\,du\\
&  \leq c_{x}^{p}\,|B_{X}(0,2)|\cdot\sum_{k=0}^{\infty}2^{kq}\,e^{-2^{2k-1}%
\,p} <\infty,
\end{align*}
where $q\geq n$ is as in assumption (H.1).
\end{proof}

\subsection{General properties of the $d_{X}$-Gaussian function}

\label{subsec.dXGauss} The last ingredient for the Levi method in
\cite{BBLUMemoir} is the validity of several `structural' properties for the
Gaussian-type function $\mathbf{E}$ in \eqref{eq.defdXGaussE} (see
Prop.\,10.11 and Cor.\,10.12 in \cite{BBLUMemoir}). The next proposition shows
that all the needed properties are satisfied also in our context.

\begin{proposition}
\label{thm.propertiesGdX} Keeping the assumptions and notation of Theorem
\ref{thm.thmVariableCoeff}, and letting $\mathbf{E}=\mathbf{E}(x,y,t)$ denote
the Gaussian-type function defined in \eqref{eq.defdXGaussE}, the
fol\-low\-ing facts hold.

\begin{enumerate}
\item There exists a constant $\mathbf{c}>0$ such that
\begin{equation}
\label{eq.estimEbeta}\mathbf{E}(x,y,t)\leq\mathbf{c}\,\beta^{q/2}%
\,\mathbf{E}(x,y,\beta t),
\end{equation}
for every $x,y\in\mathbb{R}^{n}$, $t>0$ and $\beta\geq1$. \vspace*{0.04cm}

\item For every fixed $\mu\geq0$, there exists $\mathbf{c}=\mathbf{c}_{\mu}>0$
such that
\begin{equation}
\label{eq.estimEmudfract}\left(  \frac{d_{X}(x,y)^{2}}{t}\right)  ^{\mu
}\,\mathbf{E}(x,y,\lambda t) \leq\mathbf{c}_{\mu}\,\lambda^{\mu}%
\,\mathbf{E}(x,y,2\lambda t),
\end{equation}
for every $x,y\in\mathbb{R}^{n}$, $t>0$ and $\lambda>0$. \vspace*{0.04cm}

\item For every fixed $\varepsilon>0$ and $\mu\geq0$, there exists
$\mathbf{c}=\mathbf{c}_{\mu,\varepsilon}>0$ such that
\begin{equation}
\label{eq.estimEmut}t^{-\mu}\,\mathbf{E}(x,y,t)\leq\mathbf{c}_{\mu
,\varepsilon},
\end{equation}
for every $x,y\in\mathbb{R}^{n}$ and $t>0$ satisfying $d(x,y)^{2}+t
\geq\varepsilon$. \vspace*{0.04cm}

\item There exists a constant $\mathbf{c}_{0}>0$ such that, for every $T>0$,
one has
\begin{equation}
\label{eq.estimEexpd}\mathbf{E}(x,y,t)\,\exp\left(  \mu d_{X}(0,y)^{2}\right)
\leq\mathbf{c}_{0}\mathbf{E}(x,y,2t)\,\exp\left(  2\mu d_{X}(0,x)^{2}\right)
,
\end{equation}
for every $x,y\in\mathbb{R}^{n}$, $0<t\leq T$ and $0\leq\mu\leq1/(4T)$.
\vspace*{0.04cm}

\item For every $\kappa_{1},\,\kappa_{2}>0$, there exist $\kappa_{0}%
,\,\Theta>0$ such that
\begin{equation}
\label{eq.reprodE}\int_{\mathbb{R}^{n}}\mathbf{E}(x,\zeta,\kappa
_{1}t)\,\mathbf{E} (\zeta,y,\kappa_{2}t)\,d\zeta\leq\Theta\,\mathbf{E}%
(x,y,\kappa_{0}t),
\end{equation}
for every $x,y\in\mathbb{R}^{n}$ and $t>0$. \vspace*{0.04cm}

\item There exists a constant $\boldsymbol{\sigma}>0$ such that
\begin{equation}
\label{eq.integralEbound}\int_{\mathbb{R}^{n}}\mathbf{E}(x,y,\kappa t)\,d
y\leq\boldsymbol{\sigma},
\end{equation}
for every $x\in\mathbb{R}^{n}$ and $t,\,\kappa>0$.
\end{enumerate}
\end{proposition}

\begin{proof}
(1)\thinspace\thinspace Since $\beta\geq1$, by using
\eqref{eq.doublingdXneededMem} we have
\begin{align*}
\mathbf{E}(x,y,t)  &  =\frac{|B_{X}(x,\sqrt{\beta t})|} {|B_{X}(x,\sqrt{t})|}
\cdot\frac{1}{|B_{X}(x,\sqrt{\beta t})|}\, \exp\left(  -\frac{d_{X}(x,y)^{2}%
}{t}\right) \\
&  \leq\gamma\beta^{q/2}\,\frac{1}{|B_{X}(x,\sqrt{\beta t})|}\, \exp\left(
-\frac{d_{X}(x,y)^{2}}{t}\right) \\
&  \leq\gamma\beta^{q/2}\,\mathbf{E}(x,y,\beta t),
\end{align*}
and this is \eqref{eq.estimEbeta}. \medskip

(2)\thinspace\thinspace Since $\mu\geq0$, we have
\[
M_{\mu}:=\sup_{\tau\in\lbrack0,\infty)}\tau^{\mu}\,e^{-\tau/2}\in(0,\infty).
\]
As a consequence, taking $\tau:=d_{X}(x,y)^{2}/(\lambda t)\geq0$, we obtain
\begin{align*}
\left(  \frac{d_{X}(x,y)^{2}}{t}\right)  ^{\mu}\,  &  \mathbf{E}(x,y,\lambda
t)=\frac{\lambda^{\mu}} {|B_{X}(x,\sqrt{\lambda t})|}\,\tau^{\mu}\,e^{-\tau}
\leq M_{\mu}\,\frac{\lambda^{\mu}}{|B_{X}(x,\sqrt{\lambda t})|}\,e^{-\tau
/2}\\[0.08cm]
&  (\text{using \eqref{eq.doublingdXneededMem} with $r=\sqrt{\lambda t}$ and
$M=\sqrt{2}$})\\
&  \leq\gamma\,2^{q/2}M_{\mu}\,\frac{\lambda^{\mu}} {|B_{X}(x,\sqrt{2\lambda
t})|} \,e^{-\tau/2}=\mathbf{c}_{\mu}\,\lambda^{\mu} \,\mathbf{E}(x,y,2\lambda
t),
\end{align*}
which is \eqref{eq.estimEmudfract}. \medskip

(3)\,\,In order to prove \eqref{eq.estimEmut}, we distinguish two cases.

\begin{itemize}
\item[(i)] $t>\varepsilon/2$. In this case, using \eqref{eq.lowerboundvolMem}
and the definition of $\mathbf{E}$, we get
\begin{align*}
t^{-\mu}\,\mathbf{E}(x,y,t)  &  \leq\frac{1}{\omega} \,t^{-\mu-q/2}%
\,\exp\left(  -\frac{d_{X}(x,y)^{2}}{t}\right)  \leq\frac{1}{\omega}%
\,t^{-\mu-q/2}\\
&  \leq\frac{1}{\omega}\,(\varepsilon/2)^{-\mu-q/2} =:\mathbf{c}%
_{\mu,\varepsilon}^{(1)}.
\end{align*}

\item[(ii)] $0<t\leq\varepsilon/2$. In this case, reminding that
$d_{X}(x,y)^{2}+t\geq\varepsilon$, we get
\[
d_{X}(x,y)^{2}\geq\varepsilon/2.
\]
From this, using again \eqref{eq.lowerboundvolMem} and the definition of
$\mathbf{E}$, we obtain
\begin{align*}
t^{-\mu}\,\mathbf{E}(x,y,t)  &  \leq\frac{1}{\omega}\,t^{-\mu-q/2}\,
\exp\left(  -\frac{d(x,y)^{2}}{t}\right)  \leq\frac{1}{\omega}\,t^{-\mu-q/2}
\,e^{-\varepsilon/(2t)}\\
&  \leq\frac{1}{\omega}\,\sup_{\tau\geq0}\big( \tau^{-\mu-q/2}\,
e^{-\varepsilon/(2\tau)}\big) =:\mathbf{c}_{\mu,\varepsilon}^{(2)}.
\end{align*}

\end{itemize}

Collecting the two cases, we infer that \eqref{eq.estimEmut} holds with
$\mathbf{c}_{\mu,\varepsilon}:= \max\{\mathbf{c}_{\mu,\varepsilon}%
^{(1)},\mathbf{c}_{\mu,\varepsilon}^{(2)}\}$. \medskip

(4)\thinspace\thinspace First of all, using triangle's inequality for the
distance $d_{X}$, we have
\[
d_{X}(0,y)^{2}\leq2d_{X}(x,y)^{2}+2d_{X}(0,x)^{2}\qquad\forall\,\,x,y\in
\mathbb{R}^{n};
\]
as a consequence, for every $x,y\in\mathbb{R}^{n}$ and every $t>0$ we obtain
\begin{align*}
&  \exp\left(  \mu d_{X}(0,y)^{2}\right)  \,\mathbf{E}(x,y,t)\\
&  \quad\leq\exp\left(  2\mu d_{X}(0,x)^{2}\right)  \cdot\frac{1}%
{|B_{X}(x,\sqrt{t})|}\,\exp\left(  d_{X}(x,y)^{2}(2\mu-1/t)\right) \\
&  \quad(\text{using \eqref{eq.doublingdXneededMem} with $r=\sqrt{t}$ and
$M=\sqrt{2}$})\\[0.1cm]
&  \quad\leq\gamma\,2^{q/2}\,\exp\left(  2\mu d_{X}(0,x)^{2}\right)
\cdot\frac{1}{|B_{X}(x,\sqrt{2t})|}\,\exp\left(  d_{X}(x,y)^{2}(2\mu
-1/t)\right)  =(\bigstar).
\end{align*}
We now observe that, if $T>0$ is arbitrarily fixed, one has
\begin{equation}
\label{eq.estimexpmudelta}2\mu-\frac{1}{t}\leq\frac{1}{2T}-\frac{1}{t}%
\leq-\frac{1}{2t},
\end{equation}
for every $0<t\leq T$ and $0\leq\mu\leq1/(4T)$; using
\eqref{eq.estimexpmudelta}, we then get
\begin{align*}
(\bigstar)  &  \leq\gamma\,2^{q/2}\,\exp\left(  2\mu d_{X}(0,x)^{2}\right)
\cdot\frac{1}{|B_{X}(x,\sqrt{2t})|}\,\exp\left(  -\frac{d_{X}(x,y)^{2}}%
{2t}\right) \\[0.1cm]
&  =\mathbf{c}_{0}\exp\left(  2\mu d_{X}(0,x)^{2}\right)  \,\mathbf{E}%
(x,y,2t),
\end{align*}
which is \eqref{eq.estimEexpd}. \medskip

(5)\thinspace\thinspace We first observe that, setting $\widehat{\kappa}
:=\max\{\kappa_{1},\kappa_{2}\}$, we deduce from \eqref{eq.estimEbeta} that
there exists a constant $\gamma=\gamma_{\kappa_{1},\kappa_{2}}>0$ such that
\begin{equation}
\label{eq.estimEiEkappazero}\mathbf{E}(x,\zeta,\kappa_{i}t)\leq\gamma
\,\mathbf{E}(x,\zeta,\widehat{\kappa}t) \qquad(i = 1,2)
\end{equation}
for every $x,\zeta\in\mathbb{R}^{n}$ and $t>0$. On the other hand, if
$(t_{0},x_{0})\in\mathbb{R}^{1+n}$ is arbitrarily fixed, we can exploit
estimate \eqref{eq.uniformGaussianFrozenEE}: there exist $\nu_{\Lambda
},\,\kappa_{\Lambda}>0$ such that, for any $x,y,\zeta\in$\emph{$\mathbb{R}%
^{n}$ }and any $t>0$, one has
\begin{equation}
\label{eq.estimEkappazeroGamma}%
\begin{split}
&  \mathbf{E}(x,\zeta,\widehat{\kappa}t)\leq\nu_{\Lambda}\, \Gamma
_{(t_{0},x_{0})}(\alpha t,x;0,\zeta)\quad\text{and}\\[0.1cm]
&  \mathbf{E}(\zeta,y,\widehat{\kappa}t)\leq\nu_{\Lambda}\, \Gamma
_{(t_{0},x_{0})}(\alpha t,\zeta;0,y),
\end{split}
\end{equation}
where $\alpha:=\widehat{\kappa}\cdot\kappa_{\Lambda}>0$ and $\Gamma
_{(t_{0},x_{0})}$ is the global heat kernel of the constant coefficient
operator \eqref{eq.OpFrozen}. Gathering
\eqref{eq.estimEiEkappazero}-\eqref{eq.estimEkappazeroGamma}, and taking into
account the `reproduction pro\-per\-ty' of $\Gamma_{(t_{0},x_{0})}$ (see
\eqref{reproduction prop Gamma cost}), we obtain
\begin{align*}
&  \int_{\mathbb{R}^{n}}\mathbf{E}(x,\zeta,\kappa_{1}t)\, \mathbf{E}%
(\zeta,y,\kappa_{2}t)d\zeta\\
&  \quad\leq(\gamma\nu_{\Lambda})^{2}\int_{\mathbb{R}^{n}}\Gamma_{(t_{0}%
,x_{0})} (\alpha t,x;0,\zeta)\,\Gamma_{(t_{0},x_{0})}(\alpha t,\zeta
;0,y)\,d\zeta\\
&  \quad(\text{since $\Gamma_{(t_{0},x_{0})}(\alpha t,\zeta;0,y)=
\Gamma_{(t_{0},x_{0})}(0,\zeta;-\alpha t,y)$, see Theorem
\ref{thm.existenceGammaprop}-(b)})\\
&  \quad= (\gamma\nu_{\Lambda})^{2}\int_{\mathbb{R}^{n}}\Gamma_{(t_{0},x_{0})}
(\alpha t,x;0,\zeta)\,\Gamma_{(t_{0},x_{0})}(0,\zeta;-\alpha t,y)\,d\zeta\\
&  \quad=c\,\Gamma_{(t_{0},x_{0})}(\alpha t,x,-\alpha t,y)=:(\bigstar),
\end{align*}
where $c:=(\gamma\nu_{\Lambda})^{2}$. Finally, using once again estimate
\eqref{eq.uniformGaussianFrozenEE} we get
\[
(\bigstar)\leq(c\,\nu_{\Lambda})\,\mathbf{E}(x,y,2\alpha\kappa_{\Lambda}t),
\]
and this gives \eqref{eq.reprodE}. \medskip

(6)\thinspace\thinspace If $(t_{0},x_{0})\in\mathbb{R}^{1+n}$ is
ar\-bi\-tra\-rily fixed, we know from estimate
\eqref{eq.uniformGaussianFrozenEE} that there exist constants $\nu_{\Lambda
},\,\kappa_{\Lambda}>0$ such that, for any $x,y\in\mathbb{R}^{n}$ and $t>0$,
one has
\[
\mathbf{E}(x,y,\kappa t)\leq\nu_{\Lambda}\Gamma_{(t_{0},x_{0})}(\alpha
t,x;0,y),
\]
where $\alpha:=\kappa\cdot\kappa_{\Lambda}$. As a consequence, by Theorem
\ref{thm.existenceGammaprop}-(d) we have
\[
\int_{\mathbb{R}^{n}}\mathbf{E}(x,y,\kappa t)\,d y \leq\nu_{\Lambda}
\int_{\mathbb{R}^{n}}\Gamma_{(t_{0},x_{0})}(\alpha t,x;0,y)\,d y=\nu_{\Lambda
},
\]
and this is exactly \eqref{eq.integralEbound}.
\end{proof}

We conclude this section with the

\begin{proof}
[Proof of Theorem \ref{thm.thmVariableCoeff}]The proof follows \emph{verbatim}
the arguments in \cite[Part II]{BBLUMemoir}, using Theorem
\ref{thm.mainParabolicCase}, Corollary \ref{cor.estimconE}, Proposition
\ref{thm.propdXnostre}, Lemma \ref{lem.sumexpdXhom}, Proposition
\ref{thm.propertiesGdX}.
\end{proof}

\section{Scale-invariant Harnack inequalities}

\label{sec.HarnackH} The aim of this last section is to prove a \emph{scale
invariant} Harnack inequality for the variable coefficient operator
$\mathcal{H}$ introduced in \eqref{H var}, that is,
\[
\mathcal{H}=\sum_{i,j=1}^{m}a_{i,j}(t,x)X_{i}X_{j}-\partial_{t}.
\]
In dealing with parabolic differential operators, the beautiful connection
between Harnack-type inequalities and the availability of two-sided Gaussian
bounds for the associated heat kernel was firstly pointed out by Nash
\cite{Nash}. Twenty years later, the approach of Nash was rigorously
implemented by Fabes and Stroock \cite{FabStr}, also inspired by some ideas of
Krylov and Safonov, see \cite{KS}.

As already explained in the Introduction, however, here we do not prove the
Harnack inequality for $\mathcal{H}$ by using the global two-sides Gaussian
estimates of its associated heat kernel $\Gamma$, since a much simpler
approach is possible in our context. Namely, we derive our result from its
analog proved in \cite{BoUg} in Carnot groups, using the global lifting result
in Theorem \ref{thm.liftingGroupG}. \medskip

Throughout the sequel, we tacitly inherit all the definitions and notation
in\-tro\-du\-ced so far. In particular, $X_{1},\ldots,X_{m}$ are smooth vector
fields in $\mathbb{R}^{n}$ sa\-ti\-sfy\-ing assumptions (H.1), (H.2), (H.3),
and $\mathbb{G},\,\widehat{X}:=\{\widehat{X}_{1},\ldots,\widehat{X}_{m}\}$ are
as in Theorem \ref{thm.liftingGroupG}. We remind that $\mathbb{G}$ is a Carnot
group whose underlying manifold is $\mathbb{R}^{N}$, where
\[
\text{$N=\mathrm{dim}\big( \mathrm{Lie}(X_{1},\ldots,X_{m})\big) =n+p$
\quad(for some $p\geq1$)}.
\]
Accordingly, we denote the points $z\in\mathbb{R}^{N}$ by
\[
z=(x,\xi),\quad\text{with $x\in\mathbb{R}^{n}$ and $\xi\in\mathbb{R}^{p}$}.
\]
We then introduce the following function space.

\begin{definition}
\label{def.spaceC2} Let $\Omega\subseteq\mathbb{R}^{1+n}$ be an open set. We
define $\mathfrak{C}_{X}^{2}(\Omega)$ as the space of functions $u:\Omega
\rightarrow\mathbb{R}$ satisfying the following properties:

\begin{enumerate}
\item $u$ is continuous on $\Omega$;

\item the map $u(t,\cdot)$ has intrinsic-derivatives along the $X_{i}$'s at
every point of its domain, and $X_{i}u(t,\cdot)$ is continuous for fixed $t$;

\item the map $u(\cdot,x)$ has derivative with respect to $t$, and
$\partial_{t} u(\cdot,x)$ is continuous for fixed $x$;

\item for every fixed $1\leq i\leq m$, the map $X_{i}u(t,\cdot)$ has
intrinsic-de\-ri\-va\-ti\-ves along the $X_{j}$'s at every point of its
domain, and $X_{j}X_{i}u(t,\cdot)$ is continuous for fixed $t$.
\end{enumerate}
\end{definition}

We can now state the main result of this section:

\begin{theorem}
[Parabolic Harnack inequality]\label{thm.Harnack} Let the assumptions and the
no\-ta\-tion of Theorem \ref{thm.thmVariableCoeff} do apply. Moreover, let
\[
\text{$r_{0}>0,\,0<h_{1} <h_{2}<1$ and $\gamma\in(0,1)$}%
\]
be fixed. Then, there exists a constant $M>0$, only depending on $r_{0}%
,h_{1},h_{2}$ and $\gamma$, such that, for every $(t_{0},x_{0})\in
\mathbb{R}^{1+n}$, $r\in(0,r_{0}]$ and
\[
u\in\mathfrak{C}_{X}^{2}\big(  (t_{0}-r^{2},\,t_{0})\times B_{X}%
(x_{0},r)\big)
\cap C\big(  [t_{0}-r^{2},\,t_{0}]\times\overline{B_{X}(x_{0},r)}\big)
\]
satisfying $\mathcal{H}u=0$ and $u\geq0$ on $(t_{0}-r^{2},t_{0})\times
B_{X}(x_{0},r)$, we have
\begin{equation}
\label{eq.HarnackH}\sup_{(t_{0}-h_{2}r^{2},\,t_{0}-h_{1}r^{2})\times
B_{X}(x_{0},\gamma r)} u \leq M\,u(t_{0},x_{0}).
\end{equation}

\end{theorem}

\begin{remark}
On account of \cite[Thm.\,14.4]{BBLUMemoir}, if $f\in C_{X,\mathrm{loc}%
}^{\alpha}$ and $u$ is any $\mathfrak{C}^{2}_{X}$-solution of $\mathcal{H}u =
f$, then $u$ actually belongs to $C^{2,\alpha}_{X,\mathrm{loc}}$.
\end{remark}

In order to prove Theorem \ref{thm.Harnack}, we need some preliminary lemmas.

\begin{lemma}
\label{lem.regullift} Let $\Omega\subseteq\mathbb{R}^{1+n}$ be an open set,
and let $u\in\mathfrak{C}_{X}^{2}(\Omega)$. We denote by $\pi:\mathbb{R}^{N}
\equiv\mathbb{R}^{n}\times\mathbb{R}^{p}\rightarrow\mathbb{R}^{n}$ the
projection of $\mathbb{R}^{N}$ onto $\mathbb{R}^{n}$, and we define
\[
v(t,z):=u\left(  t,\pi(z)\right)  .
\]
Then, $v$ belongs to $\mathfrak{C}^{2}_{\widehat{X}}(\Omega\times
\mathbb{R}^{p})$ and, for every $(t,z)\in\Omega\times\mathbb{R}^{p}$, we have
the identities

\begin{itemize}
\item[(a)] $\partial_{t}v(t,z)=(\partial_{t}u)(t,\pi(z))$;

\item[(b)] $\widehat{X}_{i} v(t,z) = (X_{i}u)(t,\pi(z))$ for every $1\leq
i\leq m $;

\item[(c)] $\widehat{X}_{i}\widehat{X}_{j}v(t,z) = (X_{i}X_{j}u)(t,\pi(z))$
for every $1\leq i,j\leq m$.
\end{itemize}
\end{lemma}

\begin{proof}
First of all, since $u\in\mathfrak{C}_{X}^{2}(\Omega)$, it is immediate to see that

\begin{itemize}
\item[(i)] $v$ is continuous on $\Omega\times\mathbb{R}^{p}$;

\item[(ii)] for any fixed $z=(x,\xi)$, the map $v(\cdot,z)$ is differentiable
with respect to $t$ at any point of its domain, and (a) holds.
\end{itemize}

In particular, from (a) we recognize that $\partial_{t} v(\cdot,z)$ is
con\-ti\-nuo\-us for every fixed $z$. \vspace*{0.05cm}

Next, let $i,j\in\{1,\ldots,m\}$ be fixed. Since, by Theorem
\ref{thm.liftingGroupG}, we have
\[
\widehat{X}_{i}=X_{i}+\textstyle\sum_{k=1}^{p}r_{k,i}(x,\xi)\partial_{\xi_{k}}%
\]
(for smooth functions $r_{k,i}$), it is easy to recognize that
\begin{equation}
\label{eq.projectionExp}\pi\big(\exp(t\widehat{X}_{i})\big)(z) = \exp
(tX_{i})\big(\pi(z)\big)\qquad\forall\,\,z\in\mathbb{R}^{N},\,t\in\mathbb{R}.
\end{equation}
By \eqref{eq.projectionExp}, and recalling the very definition of $v$, we then
get
\[
v\big(t,\exp(s\widehat{X}_{i})(z)\big)  = u\big(t,\pi\big(
\exp(s\widehat{X}_{i})(z)\big)\big) = u\big(t,\exp(sX_{i})(\pi(z))\big)
\]
(for all $s\in\mathbb{R},\,(t,z)\in\Omega\times\mathbb{R}^{p}$). From this,
since $u\in\mathfrak{C}_{X}^{2}(\Omega)$, we immediately deduce that
$v(t,\cdot)$ has intrinsic-derivative along $\widehat{X}_{i}$, and
\begin{equation}
\label{eq.hatXiXi}%
\begin{split}
\widehat{X}_{i}v\left(  t,z\right)   &  = \frac{d}{ds}\Big|_{s =
0}v\big(t,\exp(s\widehat{X}_{i})(z)\big)\\
&  =\frac{d}{ds}\Big|_{s = 0} u\big(t,\exp(sX_{i})(\pi(z))\big) = \left(
X_{i}u\right)  \left(  t,\pi(z)\right)  .
\end{split}
\end{equation}
Starting from \eqref{eq.hatXiXi}, and using once again
\eqref{eq.projectionExp}, we also have
\begin{align*}
\widehat{X}_{i}v\big(t,\exp(s\widehat{X}_{j})(z)\big)  &  = (X_{i}%
u)\big(t,\pi\big(\exp(s\widehat{X}_{j})(z)\big)\big)\\
&  =(X_{i}u)\big(t,\exp(sX_{i})(\pi(z))\big)
\end{align*}
(for all $s\in\mathbb{R},\,(t,z)\in\Omega\times\mathbb{R}^{p}$); from this,
since $u\in\mathfrak{C}_{X}^{2}(\Omega)$, we infer that $v(t,\cdot)$ has
intrinsic-derivatives {up to second order} along the $\widehat{X}_{k}$'s,
which are given by
\begin{equation}
\label{eq.hatXiXisecond}%
\begin{split}
\widehat{X}_{j}\widehat{X}_{i}v(t,z)  &  =\frac{d}{ds}\Big|_{s = 0}
\widehat{X}_{i}v\big(t,\exp(s\widehat{X}_{j})(z)\big)\\[0.15cm]
&  =\frac{d}{ds}\Big|_{s = 0} (X_{i}u)\big(t,\exp(sX_{i})(\pi(z))\big) =(X_{j}%
X_{i}u)(t,\pi(z)),
\end{split}
\end{equation}
and thus (b)-(c) hold. In particular, from
\eqref{eq.hatXiXi}-\eqref{eq.hatXiXisecond} (and since $u\in\mathfrak{C}%
_{X}^{2}(\Omega)$) we re\-cog\-ni\-ze that $\widehat{X}_{i}v(t,\cdot
),\,\widehat{X}_{j}\widehat{X}_{i}v(t,\cdot) $ are con\-ti\-nuo\-us for all
fixed $t$.
\end{proof}

\begin{lemma}
\label{lem.liftHolder} Let $\Omega\subseteq\mathbb{R}^{1+n}$ be an open set,
and let $f\in{C}_{X}^{\alpha}(\Omega)$ \emph{(}for a suitable $\alpha\in
(0,1)$\emph{)}. We let $\pi:\mathbb{R}^{N}\rightarrow\mathbb{R}^{n}$ be as in
Lemma \ref{lem.regullift}, and we define
\[
\widehat{f}:\Omega\times\mathbb{R}^{p}\rightarrow\mathbb{R},\qquad
\widehat{f}(t,z):=f(t,\pi(z)).
\]
Then, $\widehat{f}$ belongs to ${C}_{\widehat{X}}^{\alpha}(\Omega
\times\mathbb{R}^{p})$, and
\begin{equation}
\label{eq.normhatfHolder}\|\widehat{f}\|_{\alpha,\,\Omega\times\mathbb{R}^{p}}
\leq\|f\|_{\alpha,\,\Omega}.
\end{equation}

\end{lemma}

\begin{proof}
First of all we observe that, since $f\in{C}_{X}^{\alpha}(\Omega)$, one has
\begin{equation}
\label{eq.hatfbd}\sup_{\Omega\times\mathbb{R}^{p}}|\widehat{f}|=\sup_{\Omega
}|f|<\infty.
\end{equation}
Moreover, since we know from Remark \ref{rem.liftingdXdhatX} that
\[
d_{\widehat{X}}(z,w)\geq d_{X}(\pi(z),\pi(w))\qquad\forall\,z,w\in
\mathbb{R}^{N}\equiv\mathbb{R}^{n}\times\mathbb{R}^{p},
\]
we also have, for any $(t,z),\,(s,w)\in\Omega\times\mathbb{R}^{p}$:
\begin{equation}
\label{eq.Holderhatf}%
\begin{split}
|\widehat{f}(t,z)-\widehat{f}(s,w)|  &  = |f(t,\pi(z))-f(s,\pi(w))|\\
&  \leq\Vert f\Vert_{\alpha,\,\Omega}\left(  d_{X}(\pi(z),\pi(w))^{\alpha}+
|t-s|^{\alpha/2}\right) \\
&  \leq\Vert f\Vert_{\alpha,\,\Omega}\left(  d_{\widehat{X}}(z,w)^{\alpha}+
|t-s|^{\alpha/2}\right)  .
\end{split}
\end{equation}
Gathering \eqref{eq.hatfbd} and \eqref{eq.Holderhatf}, we immediately obtain \eqref{eq.normhatfHolder}.
\end{proof}

Thanks to Lemmas \ref{lem.regullift} and \ref{lem.liftHolder} we can now
deduce the announced scale invariant Harnack inequality for $\mathcal{H}$ from
the analogous result holding in Carnot groups.

\begin{proof}
[Proof of Theorem \ref{thm.Harnack}]Let $(t_{0},x_{0})\in\mathbb{R}%
^{1+n},\,r\in(0,r_{0}]$ and $u$ be as in the statement of the theorem.
Denoting by $\pi:\mathbb{R}^{N}\rightarrow\mathbb{R}^{n}$ the projection of
$\mathbb{R}^{N}$ onto $\mathbb{R}^{n}$, we set
\[
v:[t_{0}-r^{2},t_{0}]\times\overline{B_{X}(x_{0},r)} \times\mathbb{R}%
^{p}\rightarrow\mathbb{R},\qquad v(t,z):=u(t,\pi(z)).
\]
Moreover, we consider the variable coefficient operator
\[
\widehat{\mathcal{H}}:=\textstyle\sum_{i,j=1}^{m}\widehat{a}_{i,j}(t,z)
\widehat{X}_{i}\widehat{X}_{j}-\partial_{t},
\]
where $\widehat{A}(t,z):=(\widehat{a}_{i,j}(t,z))_{i,j=1}^{m}$ is the $m\times
m$ matrix of functions defined as
\[
\widehat{a}_{i,j}(t,z):=a_{i,j}(t,\pi(z))\qquad(\text{for $(t,z)\in
\mathbb{R}^{1+N}$}).
\]
Using Lemma \ref{lem.liftHolder}, and taking into account properties (i)-(ii)
of the matrix $A(t,x)$, we deduce that $\widehat{A}(t,z)$ satisfies the
following analogous properties:

\begin{itemize}
\item[(a)] $\widehat{a}_{i,j} \in{C}_{\widehat{X}}^{\alpha}(\mathbb{R}^{1+N})$
for every $1\leq i,j\leq m$;

\item[(b)] for every $(t,z)\in\mathbb{R}^{1+N}$ one has $\widehat{A}%
(t,z)\in\mathcal{M}_{\Lambda}$.
\end{itemize}

On the other hand, setting $z_{0}:=(x_{0},0)$, by Remark
\ref{rem.liftingdXdhatX} we have
\[
\pi\left(  B_{\widehat{X}}(z_{0},r)\right)  =B_{X}(x_{0},\rho)\quad
\text{and}\quad\pi\left(  \overline{B_{\widehat{X}}(z_{0},r)}\right)
\subseteq\overline{B_{X}(x_{0},r)},
\]
as a consequence, from Lemma \ref{lem.regullift} we deduce the following facts:

\begin{itemize}
\item[(c)] $v\in\mathfrak{C}_{\widehat{X}}^{2}\left(  (t_{0}-r^{2}%
,t_{0})\times B_{\widehat{X}}(z_{0},r)\right)  \cap C\left(  [t_{0}%
-r^{2},t_{0}]\times\overline{B_{\widehat{X}}(z_{0},r)}\right)  $;
\vspace*{0.15cm}

\item[(d)] for every $(t,z)\in(t_{0}-r^{2},t_{0})\times B_{\widehat{X}}%
(z_{0},r)$, one has
\begin{align*}
\widehat{\mathcal{H}}v(t,z)  &  =\sum_{i,j=1}^{m}\widehat{a}_{i,j}
(t,z)\widehat{X}_{i}\widehat{X}_{j}v(t,z)-\partial_{t}v(t,z)\\[0.08cm]
&  =\sum_{i,j=1}^{m}a_{i,j}(t,\pi(z))(X_{i}X_{j}u)(t,\pi(z))-(\partial_{t}u)
(t,\pi(z))\\[0.08cm]
&  =(\mathcal{H}u)(t,\pi(z))=0;
\end{align*}

\item[(e)] $v\geq0$ on $(t_{0}-r^{2},t_{0})\times B_{\widehat{X}}(z_{0},r)$.
\end{itemize}

Gathering (a)-(e), and reminding that the $\widehat{X}_{i}$'s are
Lie-generators of the Lie algebra of $\mathbb{G}$, we can apply
\cite[Thm.\,1.1]{BoUg}: there exists a constant $M>0$, only depending on
$h_{1},h_{2},\gamma$ and $r_{0}$, such that
\begin{equation}
\label{eq.HarnacksuG}\sup_{(t_{0}-h_{2}r^{2},\,t_{0}-h_{1}r^{2})\times
B_{\widehat{X}}(z_{0},\gamma r)}v\leq M\,v(t_{0},z_{0}).
\end{equation}
We finally note that, by the very definition of $v$, the above
\eqref{eq.HarnacksuG} is precisely \eqref{eq.HarnackH}. This completes the proof.
\end{proof}

Starting from Theorem \ref{thm.Harnack}, we immediately obtain (in a standard
way) a scale-in\-va\-riant Harnack inequality for the `stationary' operator
\[
\mathcal{L}:=\sum_{i,j=1}^{m}c_{i,j}(x)X_{i}X_{j}.
\]
In order to clearly state this result, we first need to introduce the
`stationary' coun\-ter\-parts of the spaces $\mathfrak{C}^{2}_{X}$ and
$C^{2,\alpha}_{X}$. We begin with the following

\begin{definition}
\label{def.spaceCX2stat} Let $U\subseteq\mathbb{R}^{n}$ be an open set. We
define $C_{X}^{2}(U)$ as the space of functions $u:U\rightarrow\mathbb{R}$
satisfying the following properties:

\begin{enumerate}
\item $u$ is continuous on $U$;

\item $u$ has continuous intrinsic-derivatives along the $X_{i}$'s at every
point of $U$;

\item $u$ for every fixed $1\leq i\leq m$, the function $X_{i}u$ has
continuous intrinsic-de\-ri\-va\-tive along the $X_{j}$'s at every point of
$U$.
\end{enumerate}
\end{definition}

We then introduce the `stationary' H\"older spaces.

\begin{definition}
\label{def.Holderstat} Let $U\subseteq\mathbb{R}^{n}$ be an open set, and let
$\alpha\in(0,1)$ be fixed. We define $C_{X}^{\alpha}(U)$ as the space of all
functions $u:U\rightarrow\mathbb{R}$ such that
\[
\sup_{U}|u|+\sup_{x\neq y\in U}\frac{|u(x)-u(y)|}{d_{X}(x,y)^{\alpha}} <
\infty.
\]

\end{definition}

With Definitions \ref{def.spaceCX2stat} and \ref{def.Holderstat} at hand, we
immediately derive the stationary coun\-ter\-part of Theorem \ref{thm.Harnack}.

\begin{theorem}
[Stationary Harnack inequality]\label{thm.HarnackStat} Let $X=\{X_{1}%
,\ldots,X_{m}\}\subseteq\mathcal{X}(\mathbb{R}^{n})$ be a family of smooth
vector fields satisfying \emph{(H.1)}-\emph{(H.2}). Moreover, let
\[
C(x)=(c_{i,j}(x))_{i,j=1}^{m}%
\]
be a $m\times m$ matrix of functions such that

\begin{itemize}
\item[(i)] $c_{i,j}\in C_{X}^{\alpha}(\mathbb{R}^{n})$ for every
$i,j=1,\ldots,m$; \vspace*{0.05cm}

\item[(ii)] there exists $\Lambda\geq1$ such that $C(x)\in\mathcal{M}%
_{\Lambda}$ for every $x\in\mathbb{R}^{1+n}$;
\end{itemize}

and let $\mathcal{L}$ be the variable coefficient operator defined as
\[
\mathcal{L}=\sum_{i,j=1}^{m}c_{i,j}(x)X_{i}X_{j}.
\]
Finally, let $r_{0}>0$ be arbitrarily fixed. \vspace*{0.08cm}

Then, there exists a constant $\mathbf{c}>0$, only depending on $r_{0}$, such
that, for every $x_{0}\in\mathbb{R}^{n}$, $r\in(0,r_{0}]$ and $u\in{C}_{X}%
^{2}(B_{X}(x_{0},3r))$ satisfying
\[
\text{$\mathcal{L}u=0$ and $u\geq0$ on $B(x_{0},3r)$},
\]
we have the following inequality
\[
\sup_{B_{X}(x_{0},r)}u\leq\mathbf{c}\,\inf_{B(x_{0},r)}u.
\]

\end{theorem}

\end{document}